\documentclass[12pt, a4paper]{amsart}
\usepackage{amsfonts, amssymb, amsmath}
\usepackage{amsthm}
\usepackage{tabularx}
\usepackage[margin=2.5cm]{geometry}
\usepackage{enumerate}
\usepackage[all]{xy}

\newcolumntype{C}[1]{>{\centering\arraybackslash}p{#1}}
\newcommand{\N}{{\mathbb N}}

\newcommand{\Q}{{\mathbb Q}}
\newcommand{\R}{{\mathbb R}}
\newcommand{\Z}{{\mathbb Z}}
\newcommand{\Pp}{{\mathbb P}}
\newcommand{\A}{{\mathbb A}}
\newcommand{\Cc}{{\mathbb C}}
\newcommand{\Oo}{\mathcal{O}}
\newcommand{\Ss}{\mathcal{S}}
\newcommand{\cen}{\mbox{Center}}
\newcommand{\se}[2]{\left\lbrace #1 \mbox{ }\vline\mbox{ } #2 \right\rbrace}

\newcommand{\mul}{\mbox{mult }}
\newcommand{\mmul}{\mbox{mult}}
\newcommand{\tl}[1]{\tilde{#1}}

\newcommand{\st}{^{\ast}}

\newcommand{\vc}[1]{\vcenter{\hbox{#1}}}

\newcommand{\w}{\mbox{wt}}
\newcommand{\gq}[1]{_{\geq #1}}

\newcommand{\ang}[1]{\langle #1\rangle}
\newcommand{\du}{^{\vee}}
\newcommand{\spc}{\mbox{Spec }}
\newcommand{\gd}{dep_{Gor}}

\newtheorem{thm}{Theorem}[section]
\newtheorem{pro}[thm]{Proposition}
\newtheorem{cor}[thm]{Corollary}
\newtheorem{lem}[thm]{Lemma}
\theoremstyle{definition}
\newtheorem{rk}[thm]{Remark}

\newtheorem{cov}[thm]{Convention}
\newtheorem*{defn}{Definition}
\newtheorem*{cl}{Claim}

\begin{document}
\title{On the factorization of three-dimensional terminal flops}
\author{Hsin-Ku Chen}
\address{School of Mathematics, Korea Institute for Advanced Study, 85 Hoegiro, Dongdaemun-gu, Seoul 02455, Republic of Korea} 
\email{hkchen@kias.re.kr}

\begin{abstract}
	We factorize three-dimensional terminal flops into a composition of divisorial contractions to points and blowing-up smooth curves.
\end{abstract}
\maketitle
\section{Introduction}
Flops are one of the typical birational maps which occurs naturally in the minimal model program. It is proved by Koll\'{a}r \cite{k} (for
three-dimensional case) and Kawamata \cite{ka} (in general) that minimal models of terminal varieties are connected by flops. As an consequence,
we know that Calabi-Yau threefolds which belong to the same birational class can be connected by flops.
On the other hands, J. A. Chen and Hacon proved that every step of three-dimensional terminal MMP can be factorize into a combination
of (inverses of) divisorial contractions to points, blowing-up smooth curves and flops, and the former two kinds of birational maps are
well-understood nowadays. Thus understanding flops becomes an important issue in three-dimensional birational geometry.\par
The simplest flops were constructed by Atiyah and Reid \cite[Part II]{re}. The construction is as follows: they consider a smooth threefold
$X$ which contains a $K$-trivial smooth rational curve $C$ with normal bundle $\Oo(-1)\oplus\Oo(-1)$ or $\Oo\oplus\Oo(-2)$ (the threefold can be
taken as a small resolution of a deformation of an $A$-type Du Val singularity). After suitably blowing-up and blowing-down smooth rational
curves, one can get another smooth threefold $X'$, such that $X-C$ and $X'-C'$ are isomorphic for some $K$-trivial smooth rational curve $C'$
on $X'$. These flops are known as Atiyah/pagoda flops.\par
In the late twentieth century rich theories about three-dimensional flops were developed. Pinkham \cite{p} and Katz-Morrison \cite{kam}
classified three-dimensional simple smooth flops (that is, smooth flops with only one flopping curve). Koll\'{a}r \cite{k} gives an explicit
local description of three-dimensional terminal flops. Nevertheless, it is still unclear that how to construct a meaningful factorization, 
as in the case of Atiyah flops or pagoda flops. In \cite{p} Pinkham described an example, which is a factorization of a simple smooth flop
with normal bundle $\Oo(1)\oplus\Oo(-3)$. But he had not developed a general theory, and his technique (computing the normal bundle sequence)
can not be applied to non-smooth flops.\par
In this paper we construct a factorization for three-dimensional terminal flops using the minimal model program. We will prove the following.
\begin{thm}\label{mthm}
	Let $X\dashrightarrow X'$ be a three-dimensional terminal $\Q$-factorial flop. Then $X\dashrightarrow X'$ can be factorize into a
	composition of divisorial contractions to points, blowing-up smooth curves, and inverses of the above maps.
\end{thm}

Combining \cite[Theorem 1.1]{ch} and \cite[Theorem 1.2]{c2}, we know that
\begin{cor}
	Each step of MMP beginning with terminal threefolds can be factorized into a composition of divisorial contractions
	to points, blowing-up smooth curves, and inverses of the above maps.
\end{cor}

To prove Theorem \ref{mthm}, we first factorize three-dimensional terminal flops into a combination of (inverses of) divisorial contractions
and simple smooth flops.

\begin{thm}\label{fthm}
	Let $X\dashrightarrow X'$ be three-dimensional terminal $\Q$-factorial flop. Then $X\dashrightarrow X'$ can be factorize into a composition of
	divisorial contractions to points, blowing-up smooth curves, simple smooth flops, and inverses of the above maps.
\end{thm}

The basic idea of the proof of Theorem \ref{fthm} is the follows: assume that $X\dashrightarrow X'$ is a flop over $W$. We first construct
a divisorial contraction $W_1\rightarrow W$ such that $W_1$ has better singularities than $W$. Let $Y$ be a $\Q$-factorizations of $W_1$
and one can run the $K_Y$-MMP over $W$. The minimal model is a $\Q$-factorization of $W$, and it may be $X$ or $X'$. Assume that the minimal
model of $Y$ over $W$ is $X$, then the birational map $Y\dashrightarrow X$ is a sequence of flips followed by a divisorial contraction.
By the result of J. A. Chen and Hacon \cite{ch} we can factorize this birational map into (inverses of) divisorial contractions and flops.\par
On the other hand, one can show that there exists $Y'$ such that either $Y\dashrightarrow Y'$ is either a flop or an isomorphism, and $X'$ is a
minimal model of $Y'$ over $W$. One can also factorize $Y'\dashrightarrow X'$ into (inverses of) divisorial contractions and flops. In this way
we get a factorization of $X\dashrightarrow X'$. If $X\dashrightarrow X'$ is not a simple smooth flop, then one can show that every flop appear
in the factorization has better singularities compared with $X\dashrightarrow X'$. Thus we can factorize $X\dashrightarrow X'$ into a
composition of (inverses of) divisorial contractions and simple smooth flops by induction on the singularity.\par
Now assume that $X\dashrightarrow X'$ is a simple smooth flop over $W$. One can first verify that the singularities of $W$ have only finitely
many possibilities. We can construct the factorization of $X\dashrightarrow X'$ as before, and every step of the factorization can be written
down explicitly. By direct computation, one can figure out that every flop appear in the factorization is better than $X\dashrightarrow X'$
(in fact, most of them are Atiyah flops). One can prove the following theorem. Here a $w$-morphism is a divisorial contraction to a point with
minimal discrepancy. Please see Section \ref{ster} for the precise definition.\par
\begin{thm}\label{ssthm}
	Let $X\dashrightarrow X'$ be a three-dimensional simple smooth flop over $W$. Then $X\dashrightarrow X'$ can be factorize into a
	composition of $w$-morphisms, blowing-up smooth curves contained in smooth loci, and inverses of the above maps.\par
	Moreover,
	\begin{enumerate}[(1)]
	\item \cite{re} Assume that $W$ has $cA$ singularities, then $X\dashrightarrow X'$ has a factorization of type $(A^{(k)})$
		for some $k\geq 1$.
	\item Assume that $W$ has $cD$ singularities, then $X\dashrightarrow X'$ has a factorization of type $(D^{(k)})$
		for some $k\geq 0$.
	\item Assume that $W$ has $cE_n$ singularities, then $X\dashrightarrow X'$ has a factorization of type $(E_n)$.
	\end{enumerate}
	In the following notation, the label $w$ means a $w$-morphism, and the label $c$ means a blowing-up a smooth curve.
	\[ (A^{(0)}): X\cong X';\quad (A^{(1)}): \vc{\xymatrix{ & Y\ar[rd]^c \ar[ld]_c & \\ X & & X' }};\quad
		(A^{(k)}):\vc{\xymatrix@C=2cm{ Y\ar[d]_c \ar@{-->}[r]^{(A^{(k-1)})} & Y'\ar[d]^c \\ X & X' } }.\]
	\[ (D^{(0)}):\vc{\xymatrix@C=2cm{ Y_{(0,1)}\ar[d]_c & Y_{(0,0)} \ar@{-->}[l]_{(A^{(1)})} \ar[d]^w &
		Y'_{(0,0)} \ar@{-->}[r]^{(A^{(1)})} \ar[d]_w & Y'_{(0,1)}\ar[d]^c \\
		Y_{(1)}\ar[d]_c & Y_{(0)} \ar@{-->}[r]^{(A)} & Y'_{(0)} & Y'_{(1)} \ar[d]^c \\ X & & & X' }};\]
	\[(D^{(k)}):\vc{\xymatrix@C=2cm{ Y_{(0,1)}\ar[d]_c & Y_{(0,0)} \ar@{-->}[l]_{(A^{(1)})} \ar[d]^w &
		Y'_{(0,0)} \ar@{-->}[r]^{(A^{(1)})} \ar[d]_w & Y'_{(0,1)}\ar[d]^c \\
		Y_{(1)}\ar[d]_c & Y_{(0)} \ar@{-->}[r]^{(D^{(k-1)})} & Y'_{(0)} & Y'_{(1)} \ar[d]^c \\ X & & & X' }}.\]
	\[ (E_6):\vc{\xymatrix@R=1.5cm{ X & Y_1\ar[l]_c & Y_{(0,2)}\ar[l]_c & Y_{(0,1,1)} \ar[l]_c & \\
		& & Y_{(0,1)} & Y_{(0,1,0)} \ar@{-->}[u]_{(A^{(1)})}\ar[l]_w & \\
		& & Y_{(0,0)} \ar@{-->}[u]_{(A^{(1)})} & Z_{(0,2)} \ar[l]_w & Z_{(0,1,1)} \ar[l]_c \\
		& & & Z_{(0,1)} & Z_{(0,1,0)} \ar@{-->}[u]_{(A^{(1)})} \ar[l]_w \\
		& & Z \ar@{-->}[d]^{(A)\mbox{ or }(D)} & Z_{(0,0)} \ar@{-->}[u]_{(A^{(1)})} \ar[l]_w & \\
		& & Z' & Z'_{(0,0)} \ar@{-->}[d]^{(A^{(1)})}\ar[l]_w & \\
		& & & Z'_{(0,1)} & Z'_{(0,1,0)} \ar@{-->}[d]^{(A^{(1)})} \ar[l]_w \\
		& & Y'_{(0,0)} \ar@{-->}[d]^{(A^{(1)})} & Z'_{(0,2)} \ar[l]_w & Z'_{(0,1,1)} \ar[l]_c \\
		& & Y'_{(0,1)} & Y'_{(0,1,0)} \ar@{-->}[d]^{(A^{(1)})}\ar[l]_w & \\
		X' & Y'_1\ar[l]_c & Y'_{(0,2)}\ar[l]_c & Y'_{(0,1,1)} \ar[l]_c & \\ }}.\]
	\[ (E_7):\vc{\xymatrix@R=0.7cm{ X & \bar{Y}\ar[l]_c & \tl{Y}\ar[l]_c & Y_{(0,2,1)} \ar[l]_c & & \\
		& & Y_{(0,2)} & Y_{(0,2,0)} \ar[l]_w \ar@{-->}[u]_{(A^{(1)})} & & \\
		& & & Y_{(0,1,1)} \ar[lu]^c & & \\
		& & Y_{(0,1)} & Y_{(0,1,0)} \ar[l]_w \ar@{-->}[u]_{(A^{(1)})}  & & \\
		& & & \bar{Z}_{(1)} \ar[lu]^w & \bar{Z}_{(0,2)} \ar[l]_w & \bar{Z}_{(0,1,1)} \ar[l]_c \\
		& & & & \bar{Z}_{(0,1)} & \bar{Z}_{(0,1,0)} \ar[l]_w \ar@{-->}[u]_{(A^{(1)})} \\
		& & & & \bar{Z}_{(0)}=Z_{(0)}\ar@{-->}[u]_{(A^{(1)})} \ar@{-->}[d]^{(A^{(1)})} & \\
		& & & & Z_{(0,1)} & Z_{(0,1,0)} \ar[l]_w \ar@{-->}[d]^{(A^{(1)})} \\
		& Y \ar@{-->}[d]^{(D)} & Y_{(0,0)} \ar[l]_w & Z_{(1)} \ar[l]_w & Z_{(0,2)} \ar[l]_w & Z_{(0,1,1)} \ar[l]_c  \\
		& Y' & Y'_{(0,0)} \ar[l]_w & Z'_{(1)} \ar[l]_w & Z'_{(0,2)} \ar[l]_w & Z'_{(0,1,1)} \ar[l]_c  \\
		& & & & Z'_{(0,1)} & Z'_{(0,1,0)} \ar[l]_w \ar@{-->}[u]_{(A^{(1)})} \\
		& & & & Z'_{(0)}=\bar{Z'}_{(0)}\ar@{-->}[d]^{(A^{(1)})} \ar@{-->}[u]_{(A^{(1)})} & \\
		& & & & \bar{Z'}_{(0,1)} & \bar{Z'}_{(0,1,0)} \ar[l]_w \ar@{-->}[d]^{(A^{(1)})} \\
		& & & \bar{Z'}_{(1)} \ar[ld]_w & \bar{Z'}_{(0,2)} \ar[l]_w & \bar{Z'}_{(0,1,1)} \ar[l]_c \\
		& & Y'_{(0,1)} & Y'_{(0,1,0)} \ar[l]_w \ar@{-->}[d]^{(A^{(1)})}  & & \\
		& & & Y'_{(0,1,1)} \ar[ld]_c & & \\
		& & Y'_{(0,2)} & Y'_{(0,2,0)} \ar[l]_w \ar@{-->}[d]^{(A^{(1)})} & & \\
		X' & \bar{Y'}\ar[l]_c & \tl{Y'}\ar[l]_c & Y'_{(0,2,1)} \ar[l]_c & & \\ }}.\]
		
	\[ (E_8^{(1)}):\vc{\xymatrix@R=0.8cm@C=0.8cm{ X & Y_1\ar[l]_c & Y_{(0,2)} \ar[l]_c & Y_{(0,1,1)} \ar[l]_c & & & & \\
		& & Y_{(0,1)} & Y_{(0,1,0)} \ar@{-->}[u]_{(A^{(1)})} \ar[l]_w & & & & \\
		& & Y_{(0,0)} \ar@{-->}[u]_{(A^{(1)})} & Z_{(0,4)} \ar[l]_w & Z_{(0,3,1)} \ar[l]_c & & & \\
		& & & Z_{(0,3)} & Z_{(0,3,0)} \ar[l]_w \ar@{-->}[u]_{(A^{(1)})} & & & \\
		& & & & Z_{(0,2,1)} \ar[lu]^c & & & \\
		& & & & Z_{(0,2,0)} \ar@{-->}[u]_{(A^{(1)})} & Z_{(0,1,1,1)} \ar[l]_c & & \\
		& & & & Z_{(0,1,1)} & Z_{(0,1,1,0)} \ar[l]_w \ar@{-->}[u]_{(A^{(1)})} & & \\
		& & & Z_{(0,1)} & Z_{(0,1,0)} \ar[l]_w \ar@{-->}[u]_{(A^{(1)})} & & & \\
		& & & Z_{(0,0)} \ar@{-->}[u]_{(A^{(1)})} & \bar{Z}_{(0,2)} \ar[l]_w & \bar{Z}_{(0,1,2)} \ar[l]_w & \bar{Z}_{(0,1,1,1)} \ar[l]_c & \\
		& & & & & \bar{Z}_{(0,1,1)} & \bar{Z}_{(0,1,1,0)} \ar[l]_w \ar@{-->}[u]_{(A^{(1)})} & \\
		& & & & \bar{Z}_{(0,1)} & \bar{Z}_{(0,1,0)} \ar[l]_w\ar@{-->}[u]_{(A^{(1)})} & & \\
		& & & & \bar{Z}_{(0,0)} \ar@{-->}[u]_{(A^{(1)})} & \tl{Z}_{(0,2)} \ar[l]_w & \tl{Z}_{(0,1,2)} \ar[l]_w & \tl{Z}_{(0,1,1,1)} \ar[l]_c \\
		& & & & & & \tl{Z}_{(0,1,1)} & \tl{Z}_{(0,1,1,0)} \ar[l]_w \ar@{-->}[u]_{(A^{(1)})} \\
		& & & & & \tl{Z}_{(0,1)} & \tl{Z}_{(0,1,0)} \ar[l]_w \ar@{-->}[u]_{(A^{(1)})} & \\
		& & & & \tl{Z}_{(0)} \ar@{-->}[d]_{(A)\mbox{ or }(D)} & \tl{Z}_{(0,0)} \ar[l]_w \ar@{-->}[u]_{(A^{(1)})} & & \\
		& & & & \tl{Z'}_{(0)} & & & } }\] 
		and the diagram from $\tl{Z'}_{(0)}$ to $X'$ is symmetric.
	\[ (E_8^{(2)}):\vc{\xymatrix@R=0.8cm@C=0.8cm{ X & Y_1\ar[l]_c & Y_{(0,2)} \ar[l]_c & Y_{(0,1,2)} \ar[l]_c & Y_{(0,1,1,1)} \ar[l]_c & & & \\
		& & & Y_{(0,1,1)} & Y_{(0,1,1,0)} \ar@{-->}[u]_{(A^{(1)})} \ar[l]_w & & & \\
		& & Y_{(0,1)} & Y_{(0,1,0)} \ar@{-->}[u]_{(A^{(1)})} \ar[l]_w & & & & \\
		& & Y_{(0,0)} \ar@{-->}[u]_{(A^{(1)})} & Z_{(0,3)} \ar[l]_w & Z_{(0,2,1)} \ar[l]_c & & & \\
		& & & Z_{(0,2)} & Z_{(0,2,0)} \ar[l]_w \ar@{-->}[u]_{(A^{(1)})} & & & \\
		& & & & Z_{(0,1,2)} \ar[lu]^w & Z_{(0,1,1,1)} \ar[l]_c & & \\
		& & & & Z_{(0,1,1)} & Z_{(0,1,1,0)} \ar[l]_w \ar@{-->}[u]_{(A^{(1)})} & & \\
		& & & Z_{(0,1)} & Z_{(0,1,0)} \ar[l]_w \ar@{-->}[u]_{(A^{(1)})} & & & \\
		& & & Z_{(0,0)} \ar@{-->}[u]_{(A^{(1)})} & \bar{Z}_{(0,2)} \ar[l]_w & \bar{Z}_{(0,1,2)} \ar[l]_w & \bar{Z}_{(0,1,1,1)} \ar[l]_c & \\
		& & & & & \bar{Z}_{(0,1,1)} & \bar{Z}_{(0,1,1,0)} \ar[l]_w \ar@{-->}[u]_{(A^{(1)})} & \\
		& & & & \bar{Z}_{(0,1)} & \bar{Z}_{(0,1,0)} \ar[l]_w\ar@{-->}[u]_{(A^{(1)})} & & \\
		& & & & \bar{Z}_{(0,0)} \ar@{-->}[u]_{(A^{(1)})} & \tl{Z}_{(0,2)} \ar[l]_w & \tl{Z}_{(0,1,2)} \ar[l]_w & \tl{Z}_{(0,1,1,1)} \ar[l]_c \\
		& & & & & & \tl{Z}_{(0,1,1)} & \tl{Z}_{(0,1,1,0)} \ar[l]_w \ar@{-->}[u]_{(A^{(1)})} \\
		& & & & & \tl{Z}_{(0,1)} & \tl{Z}_{(0,1,0)} \ar[l]_w \ar@{-->}[u]_{(A^{(1)})} & \\
		& & & & \tl{Z}_{(0)} \ar@{-->}[d]_{(A)\mbox{ or }(D)} & \tl{Z}_{(0,0)} \ar[l]_w \ar@{-->}[u]_{(A^{(1)})} & & \\
		& & & & \tl{Z'}_{(0)} & & & } }\] 
		and the diagram from $\tl{Z'}_{(0)}$ to $X'$ is symmetric.
\end{thm}
We only need to prove Theorem \ref{fthm} and Theorem \ref{ssthm}.
In Section \ref{spre} we will recall some known result about three-dimensional simple smooth flops as well as singularities and factorizations
of terminal threefolds. The proof of Theorem \ref{fthm} is given in Section \ref{sff}. In Section \ref{sflip} we describe the factorization
of flips with very simple singularities, which is useful in the rest of the article. We construct the factorization of simple smooth flop over
$cD$ points as well as smooth flops over $cA_2$ points in Section \ref{scd}, and construct the factorization of simple smooth flops over $cE$
points in Section \ref{sce}. This finishes the proof of Theorem \ref{ssthm}.\par
I want to thank Jungkai Alfred Chen and Paolo Cascini for their helpful comments. Part of this work was done while the author was visiting
the University of Cambridge DPMMS. The author want to thank the University of Cambridge DPMMS for its hospitality.
\section{Preliminaries}\label{spre}
\subsection{Conventions}
In this paper, every variety is assumed to be terminal threefolds which are projective and over the complex number.\par
A flop is diagram of small birational morphisms \[ \vc{\xymatrix{ X\ar[rd]_{f} & & X'\ar[ld]^{f'} \\ & W & }}\] such that $W$ is
$\Q$-Gorenstein, and $\rho(X/W)=\rho(X'/W)=1$. Notice that we have $K_X=f\st K_W$ and $K_{X'}={f'}\st K_W$.\par
Let $X\dashrightarrow X'$ be a three-dimensional flop having terminal singularities. By \cite{k} we know that $X$ and $X'$ have the same
singularities and $exc(X/W)\cong exc(X'/W)$. We say that $X\dashrightarrow X'$ is \emph{simple} if $exc(X/W)$ is irreducible. If $X$ is smooth (resp. Gorenstein, non-Gorenstein), then we say that $X\dashrightarrow X'$ is a smooth (resp. Gorenstein, non-Gorenstein) flop.
If $X\dashrightarrow X'$ is Gorenstein flop, we say that the flop is of type $A$, $D$ or $E$ if $W$ has $cA$, $cD$ or $cE$ singularities,
respectively.\par
A divisorial contraction is a birational morphism $Y\rightarrow X$ such that the exceptional locus is irreducible and $K_Y$ is anti-ample
over $X$.\par
Assume that $\phi:Z\dashrightarrow Y$ is a birational map and $H$ is a divisor on $Z$. If $\phi$ is an isomorphism on the
generic point of $H$, then we denote the proper transform of $H$ on $Y$ by $H_Y$.
\subsection{Simple smooth flops}
We briefly introduce the classification of three-dimensional simple smooth flops.
\begin{defn}
	Let $X\dashrightarrow X'$ be a three-dimensional simple smooth flop. Let $X_0=X$ and $C_0$ be the flopping curve. Define
	$X_{i+1}=Bl_{C_i}X_i$. The exceptional divisor $E_{i+1}$ of $X_{i+1}\rightarrow X_i$ has a ruled surface structure over $C_i$.
	If $E_{i+1}$ do not isomorphic to $\Pp^1\times \Pp^1$, then let $C_{i+1}$ be the unique negative section.\par
	The normal bundle sequence is a sequence $\{(a_i,b_i)\}_{i=0}^N$ such that $N_{C_i/X_i}=\Oo_{C_i}(a_i)\oplus\Oo_{C_i}(b_i)$ with
	$a_i\geq b_i$, and $a_N=b_N$.
\end{defn}
By \cite{p} and \cite{kam} there are seven types of simple smooth flops:\\
\[ \begin{tabular}{l|c|c}
	No. & Singularity of $W$ & normal bundle sequence \\\hline
	1 & $cA_1$ & $(-1,-1)$ \\
	2 & $cA_1$ & $(0,-2)$, ..., $(0,-2)$, $(-1,-1)$ \\
	3 & $cD_4$ & $(1,-3)$, $(-1,-2)$, $(-1,-1)$\\
	4 & $cE_6$ & $(1,-3)$, $(0,-3)$, $(-1,-2)$, $(-1,-1)$ \\
	5 & $cE_7$ & $(1,-3)$, $(0,-3)$, $(-1,-2)$, $(-1,-1)$ \\
	6 & $cE_8$ & $(1,-3)$, $(0,-3)$, $(-1,-2)$, $(-1,-1)$ \\
	7 & $cE_8$ & $(1,-3)$, $(0,-3)$, $(0,-3)$, $(-1,-2)$, $(-1,-1)$ \\
\end{tabular}\]
Case No.1 is called an \emph{Atiyah flop} and Case No.2 is called a \emph{pagoda flop}. Using the notation in Theorem \ref{ssthm}, it is
well-known that (cf. \cite[Part II]{re}) the Atiyah flop has a factorization of type $(A^{(1)})$ and a pagoda flop with a length $n$ normal
bundle sequence has a factorization of type $(A^{(n)})$.

\begin{lem}\label{lati}
	Let $X\dashrightarrow X'$ be an Atiyah flop. Let $C\subset X$ be the flopping curve and $C'\subset X'$ be the flopped curve. Let $H$ be a
	divisor on $X$. Then the following holds.
	\begin{enumerate}[(1)]
		\item Assume that $H$ intersects $C$ transversally at $m$ points, then $\mmul_{C'} H_{X'}=m$.
		\item Assume that $H.C=m$, then $H_{X'}.C'=-m$.
	\end{enumerate}
\end{lem}
\begin{proof}
	Let $\phi:Y=Bl_CX\rightarrow X$ and $\phi':Y=Bl_{C'}X'\rightarrow X'$ be the common resolution and let $E=exc(\phi)=exc(\phi')$.
	Then $H_Y$ intersects $E$ transversally at $m$ curves $l_1$, ..., $l_m$, such that $l_i$ maps bijectively to $C'$.
	Hence $\mmul_{C'} H_{X'}=m$.\par
	Now we prove (2). First assume that $m>0$ and $H$ intersects $C$ transversally at $m$ points. Let $l$ be a curve on $Y$ which maps
	bijectively to $C'$. Then we have \[H_{X'}.C'=H_Y.l+(\mmul_{C'} H_{X'}) E.l=(\phi\st H).l+(\mmul_{C'} H_{X'}) E.l=0-m=-m.\]
	In general we can write $H=A-B$ where $A$ and $B$ are ample divisors such that $B$ and some multiple of $A$ intersects $C$ transversally.
	One can easily see that $H$ has the desired property.
\end{proof}
We have the following observation for the underlying space of simple smooth flops.
\begin{lem}\label{ssflop}
	Assume that $(P\in W)$ is a germ of Gorenstein terminal threefold. Then there exists a simple smooth flop over $W$ if and only if
	$W$ is not $\Q$-factorial and there are exactly one exceptional divisor which has discrepancy one over $P$.
\end{lem}
\begin{proof}
	Assume that $X\rightarrow W$ is a flopping contraction such that $X$ is smooth and $exc(X/W)$ is irreducible. Then there is only one
	exceptional divisor $E$ over $W$ with $a(E,X)=1$, namely the divisor obtained by blowing-up the flopping curve on $X$.\par
	Conversely, if $X\rightarrow W$ is a flopping contraction with $n$ flopping curves for $n>1$, then blowing-up those curves produces $n$
	different exceptional divisors which have discrepancy one over $W$. Likewise, if $X\rightarrow W$ is a flopping contraction such that
	$X$ has a Gorenstein singular point $P$. Then there exists an exceptional divisor $F$ over $P$ such that $a(F,X)=a(F,W)=1$ (cf. \cite{c}).
	Since blowing-up a flopping curve always induces an exceptional divisor with discrepancy one, there are at least two discrepancy one
	exceptional divisors over $W$.
\end{proof}
\begin{rk}\label{minx}
	Given a flop $X\dashrightarrow X'$ over $W$ such that $X$ as well as $X'$ are $\Q$-factorial.
	Then $W$ has exactly two $\Q$-factorizations, namely $X$ and $X'$.
	Indeed, if $X_1\rightarrow W$ is a small contraction such that $X_1$ is $\Q$-factorial. Let $A_{X_1}$ be an ample divisor on $X_1$. Then
	either $A_X$ or $A_{X'}$ is ample because of the condition $\rho(X/W)=1$. Thus $X_1$ is either isomorphic to $X$ or $X'$.
\end{rk}

\subsection{Weighted blow-ups}
Let $G=\ang{\tau\mbox{ }\vline\mbox{ }\tau^r=id}$ be a cyclic group of order $r$. For any $\Z$-valued $n$-tuple $(a_1,...,a_n)$
one can define a $G$-action on $\A^n_{(x_1,...,x_n)}$ by $\tau(x_i)=\xi^{a_i}x_i$, where $\xi=e^{\frac{2\pi i}{r}}$.
We will denote the quotient space $\A^n/G$ by $\A^n_{(x_1,...,x_n)}/\frac{1}{r}(a_1,...,a_n)$.\par
Let $W\cong \A^n_{(x_1,...,x_n)}/\frac{1}{r}(a_1,...,a_n)$ be a cyclic-quotient singularity.
There is an elementary way to construct a birational morphism $Y\rightarrow W$, so called the weighted blow-up, defined as follows.\par
We write everything using the language of toric varieties. Let $N$ be the lattice $\ang{e_1,...,e_n,v}_{\Z}$,
where $\{e_1,...,e_n\}$ is the standard basis of $\R^n$ and $v=\frac{1}{r}(a_1,...,a_n)$. Let $\sigma=\ang{e_1,...,e_n}_{\R\gq0}$.
We have $W\cong \spc\Cc [N\du\cap \sigma\du]$.\par
Let $w=\frac{1}{r}(b_1,...,b_n)$ be a vector such that $b_i=\lambda a_i+k_ir$ for $\lambda\in \N$ and $k_i\in\Z$. The weighted blow-up
of $W$ with weight $w$ is the toric variety defined by the fan consists of those cones
\[\sigma_i=\ang{e_1,...,e_{i-1},w,e_{i+1},...,e_n}.\]
Let $U_i$ be the toric variety defined by the cone $\sigma_i$ and lattice $N$.
\begin{lem}\label{wbup}
	Let	\[ v'=\frac{1}{b_i}(-b_1,...,-b_{i-1},r,-b_{i+1},...,-b_n)\] and
	\[w'=\frac{1}{rb_i}(a_1b_i-a_ib_1,...,a_{i-1}b_i-a_ib_{i-1},ra_i,a_{i+1}b_i-a_ib_{i+1},...,a_nb_i-a_ib_n). \]
	Assume that $u=\frac{1}{r'}(a'_1,...,a'_n)$ is a vector such that \[\ang{e_1,...,e_n,v',w'}_{\Z}=\ang{e_1,..,e_n,u}_{\Z},\]
	then \[U_i\cong \A^n/\frac{1}{r'}(a'_1,...,a'_n).\]\par
	In particular, if $\lambda=1$, then $U_i\cong\frac{1}{b_i}(-b_1,...,-b_{i-1},r,-b_{i+1},...,-b_n)$.
\end{lem}
\begin{cor}
	Assume that the hypothesis of Lemma \ref{wbup} is satisfied. Let $x_1$, ..., $x_n$ be local coordinates of $X$ and $y_1$, ..., $y_n$ be
	local coordinates of $U_i$ defined by the cyclic quotient structure.
	The change of coordinates of $U_i\rightarrow X$ is given by $x_i=y_i^{\frac{b_i}{r}}$ and $x_j=y_jy_i^{\frac{b_j}{r}}$ for $j\neq i$.
\end{cor}
\begin{cor}
	Assume that \[S=(f_1(x_1,...,x_n)=...=f_k(x_1,...,x_n)=0)\subset W\] is a complete intersection subvariety
	and $S'$ is the proper transform of $S$ on $Y$. Assume that the exceptional locus $E$ of $S'\rightarrow S$ is irreducible and reduced.
	Then \[a(S,E)=\frac{b_1+...+b_n}{r}-\sum_{i=1}^k\w_wf_k(x_1,...,x_n)-1.\]
\end{cor}
Please see \cite[Section 2.2]{me} for the proofs of the above lemma and corollaries.\par
In this paper we consider terminal threefolds which are embedding into cyclic quotient of $\A^4$ or $\A^5$
\[ X\hookrightarrow \A^4_{(x,y,z,u)}/\frac{1}{r}(a,b,c,d)\quad\mbox{or}\quad X\hookrightarrow \A^5_{(x,y,z,u,t)}/\frac{1}{r}(a,b,c,d,e).\]
We say that $Y\rightarrow X$ is a weighted blow-up with weight $w$ if $Y$ is the proper transform of $X$ inside the weighted blow-up
of $\A^4_{(x,y,z,u)}/\frac{1}{r}(a,b,c,d)$ or $\A^5_{(x,y,z,u,t)}/\frac{1}{r}(a,b,c,d,e)$ with weight $w$.
\begin{cov}\label{cov}
	Assume that $X$ is of the above form and let $Y\rightarrow X$ be a weighted blow-up.
	The notation $U_x$, $U_y$, $U_z$, $U_u$ and $U_t$ will stand for $U_1$, ..., $U_5$ in Lemma \ref{wbup}.
\end{cov}

\subsection{Terminal threefolds}\label{ster}

\subsubsection{Singularities of terminal threefolds}
\begin{defn}
	Let $Y\rightarrow X$ be a divisorial contraction contracts a divisor $E$ to a point $P$. We say that $Y\rightarrow X$ is a 
	\emph{$w$-morphism} if $a(X,E)=\frac{1}{r_P}$, where $r_P$ is the Cartier index of $K_X$ near $P$.
\end{defn}
\begin{defn}
	The depth of a terminal singularity $P\in X$, $dep(P\in X)$, is the minimal length of the sequence
	\[ X_m\rightarrow X_{m-1}\rightarrow \cdots\rightarrow X_1\rightarrow X_0=X,\]
	such that $X_m$ is Gorenstein and $X_i\rightarrow X_{i-1}$ is a $w$-morphism for all $1\leq i\leq m$.\par
	The generalize depth of a terminal singularity $P\in X$, $gdep(P\in X)$, is the minimal length of the sequence
	\[ X_n\rightarrow X_{n-1}\rightarrow \cdots\rightarrow X_1\rightarrow X_0=X,\]
	such that $X_n$ is smooth and $X_i\rightarrow X_{i-1}$ is a $w$-morphism for all $1\leq i\leq n$. The variety $X_n$
	is called a \emph{feasible resolution} of $P\in X$.\par
	The Gorenstein depth of a terminal singularity $P\in X$, $\gd(P\in X)$, is defined by $gdep(P\in X)-dep(P\in X)$.\par
	For a terminal threefold we can define \[dep(X)=\sum_P dep(P\in X),\]
	\[gdep(X)=\sum_P gdep(P\in X)\]
	and \[ \gd(X)=\sum_P \gd(P\in X).\]
\end{defn}
\begin{rk}
	In the above definition, the existence of a sequence \[ X_m\rightarrow X_{m-1}\rightarrow \cdots\rightarrow X_1\rightarrow X_0=X,\]
	such that $X_m$ is Gorenstein follows from \cite[Theorem 1.2]{h2}. The existence of a sequence
	\[ X_n\rightarrow X_{n-1}\rightarrow \cdots\rightarrow X_1\rightarrow X_0=X,\] such that $X_n$ is smooth follows from \cite[Theorem 2]{c}.
\end{rk}
\begin{defn}
	Assume that $Y\rightarrow X$ is a $w$-morphism such that $gdep(Y)=gdep(X)-1$. Then we say that $Y\rightarrow X$ is a
	\emph{strict $w$-morphism}.
\end{defn}
\begin{rk}\label{sww}
	It has been proved in \cite[Section 5]{me2}, that if $Y\rightarrow X$ is a strict $w$-morphism, then $\gd(Y)=\gd(X)$.
\end{rk}
\begin{rk}\label{sdep}
	Assume that $X$ is terminal threefold with $\gd(X)=0$ and $P\in X$ is a singular point. By \cite[Corollary 3.4]{mo2} we know that
	$gdep(P\in X)=1$ if and only if $P$ is a $\frac{1}{2}(1,1,1)$ point (since $\gd(X)=0$ implies that $P$ is not a Gorenstein singular point).
	Also assume that $gdep(P\in X)=2$, then $P$ is a $\frac{1}{3}(1,2,1)$ point or a $cA/2$ point defined by
	\[ (xy+z^2+u^2=0)\subset\A^4_{(x,y,z,u)}/\frac{1}{2}(1,1,1,0).\] Indeed, in this case there exists a $w$-morphism
	$X_1\rightarrow X$ such that $X_1$ contains exactly one singular point which is a $\frac{1}{2}(1,1,1)$ point. Now $w$-morphisms
	between terminal threefolds are listed in \cite{h1} and \cite{h2}. One can verify that the preceding two cases are the only possibility. 
\end{rk}
We will need the following properties of the depth:
\begin{pro}[\cite{me2}, Porposition 5.1]\label{ied}$ $
	\begin{enumerate}[(1)]
	\item Assume that $Y\rightarrow X$ is a divisorial contraction between terminal and $\Q$-factorial threefolds.
		\begin{enumerate}[({1}-1)]
		\item If $Y\rightarrow X$ is a divisorial contraction to a point, then
			\[gdep(X)\leq gdep(Y)+1\mbox{ and }dep(X)\leq dep(Y)+1.\]
			If $Y\rightarrow X$ is a divisorial contraction to a curve, then
			\[gdep(X)\leq gdep(Y)\mbox{ and }dep(X)\leq dep(Y).\]
		\item $\gd(X)\geq \gd(Y)$ and the inequality is strict if the non-isomorphic locus on $X$ contains a Gorenstein singular point.
		\end{enumerate}			
	 \item Assume that $X\dashrightarrow X'$ is a flip between terminal and $\Q$-factorial threefolds. 
	 	\begin{enumerate}[({2}-1)]
	 	\item \[ gdep(X)>gdep(X')\mbox{ and }dep(X)>dep(X').\]
	 	\item $\gd(X)\leq\gd(X')$ and the inequality is strict if the non-isomorphic locus on $X'$ contains a Gorenstein singular point.
	 	\end{enumerate}
	\end{enumerate}
\end{pro}
As an easy corollary one has that:
\begin{cor}\label{feam}
	Assume that \[ X_0\dashrightarrow X_1\dashrightarrow...\dashrightarrow X_k\] is a process of the minimal model program.
	Then $\gd(X_k)\geq \gd(X_0)$. The inequality is strict if the non-isomorphic locus of $X_0\dashrightarrow X_k$ on $X_k$ contains a
	Gorenstein singular point.
\end{cor}

\begin{rk}\label{car}
	Assume that $P\in X$ is a three-dimensional cyclic quotient point, then there is only one terminal divisorial contraction $Y\rightarrow X$
	over $P$, which is a weighted blow-up in \cite[Theorem 5]{ka2}. An easy computation shows that $Y$ has only cyclic quotient singularities.\par
	Likewise, assume that $P\in X$ is a three-dimensional $cA/r$ singularity with $r>1$, then every terminal divisorial contraction
	$Y\rightarrow X$ is a weighted blow-up in \cite[Theorem 1.2 (i)]{k2}. One can easily compute that $Y$ has only $cA/r$ singularities.
\end{rk}

\subsubsection{Chen-Hacon factorization}\label{sch}

\begin{thm}[\cite{ch} Theorem 3.3]\label{chf}
	Assume that either $Y\dashrightarrow Y_1$ be a flip over $U$, or $Y\rightarrow U$ is a divisorial contraction to a curve and $Y$ is
	not Gorenstein over $U$. Then there exists a diagram
	\[\vc{\xymatrix@C=0.8cm{Z_0\ar[d] \ar@{-->}[r] & ... \ar@{-->}[r] & Z_{i} \ar[rd] & & Z_{i+1} \ar[ld]\ar@{-->}[r]
		& ... \ar@{-->}[r]  & Z_k\ar[d]\\
		Y\ar[rrrd] & & & V_i\ar[d] & & & Y_1\ar[llld] \\ & & & U & &	& }}\]
	such that $Z_0\rightarrow Y$ is a strict $w$-morphism (\cite[Corollary 5.8]{me2}), $Z_k\rightarrow Y_1$ is a divisorial contraction,
	$Z_0\dashrightarrow Z_{1}$ is a flip or a flop over $V_0$ and $Z_i\dashrightarrow Z_{i+1}$ is a flip over $V_i$ for $i>0$.
	$Y_1\rightarrow U$ is a divisorial contraction to a point if $Y\rightarrow U$ is divisorial.
\end{thm}
\begin{rk}\label{flpc}
	Notation as in the above theorem. Assume that $Y\dashrightarrow Y_1$ is a flip, $C_{Z_0}$ is the flipping/flopping curve of $Z_0\dashrightarrow Z_1$ and $C_Y$ is the image of $C_{Z_0}$ on $Y$. Then $C_Y$ is a flipping curve of $Y\dashrightarrow Y_1$. This fact follows from the construction of the diagram.
\end{rk}

\section{First factorization}\label{sff}

\subsection{General elephants}
Let $X$ be a germ of a three-dimensional terminal singularity. The general hyperplane section of $X$ (so called the \emph{general elephant}
of $X$) has only Du Val singularities, please see \cite[(6.4)]{re2} for more details. The singularity of the general elephant can be used to measure the singularities of $X$. In this subsection we will discuss how to estimate the general elephants of varieties appearing in the
Chen-Hacon factorization.
\begin{lem}
	Let $W$ be a terminal threefold. Assume that $H$ is a reduced and irreducible effective Weil divisor on $W$ such that $K_W+H$ is
	$\Q$-Cartier and $H$ has only Du Val singularities. Then $(W,H)$ is canonical.
\end{lem}
\begin{proof}
	Let \[ W_m\rightarrow W_{m-1}\rightarrow \cdots\rightarrow W_1\rightarrow W_0=W\] be a sequence of $w$-morphisms resolving the
	singularities of $Sing(W)\cap H$. Let $W_k\rightarrow \cdots\rightarrow W_m$ be a sequence of smooth blow-ups, such that $(W_k,H_{W_k})$
	is log smooth. Let $E_i=exc(W_i\rightarrow W_{i-1})$. By \cite[Lemma 2.7 (2)]{ch} we know that $a(E_i,W_{i-1},H)=0$ for all $1\leq i\leq m$
	and it is easy to compute that $a(E_i,W_{i-1},H)\leq 0$ for $m+1\leq i\leq k$. It follows that $a(E_i,W,H)\leq 0$ for all $i$. 
	This implies that $a(E_i,W,H)=0$ for all $i$, since otherwise $H$ can not have canonical singularities. Hence $(W,H)$ is canonical.
\end{proof}
\begin{cor}\label{flopge}
	Assume that $f:X\rightarrow W$ is a birational morphism between terminal threefolds. Let $H_W$ be a reduced
	and irreducible Weil divisor on $W$ such that $K_W+H$ is $\Q$-Cartier and $H$ has only Du Val singularities. 
	If $K_X+H_X=f\st(K_W+H)$, then $H_X$ has only Du Val singularities and $H_X\rightarrow H$ is a partial resolution.
\end{cor}
\begin{proof}
	We have $(W,H)$ is canonical, hence $(X,H_X)$ is canonical. \cite[Proposition 5.51]{km} implies that $H_X$ is normal.
	Hence $H_X$ has also Du Val singularities and $H_X\rightarrow H$ is a partial resolution.
\end{proof}
\begin{cor}\label{lci}
	Assume that $Y\dashrightarrow Y_1$ is a flip over $U$ and $Z_k\rightarrow Y_1$ is the last step in the factorization of Theorem \ref{chf}.
	Assume that there exists a Du Val section $H_U\in|-K_U|$ and $Z_k\rightarrow Y_1$ is a divisorial contraction to a curve $C_1$,
	then $C_1$ is smooth.
\end{cor}
\begin{proof}
	Note that $C_1$ is a flipped curve since $C_1=\cen_{Y_1} F$ and $\cen_Y F$ is a point, where $F=exc(Z_k\rightarrow Y_1)$.
	We know that $H_{Y_1}\rightarrow H_U$ is a partial resolution by Corollary \ref{flopge}. Since $C_1$ is a flipped curve,
	$H_{Y_1}.C_1<0$, so $C_1\subset H_{Y_1}$ and it is an exceptional curve of $H_{Y_1}\rightarrow H_U$. Hence $C_1$ should be smooth.
\end{proof}
Assume that either $Y\dashrightarrow Y_1$ be a flip over $U$, or $Y\rightarrow U$ is a divisorial contraction to a curve and $Y$ is
not Gorenstein over $U$. We have the following diagram
\[\vc{\xymatrix@C=0.8cm{Z_0\ar[d]_h \ar@{-->}[r] & ... \ar@{-->}[r] & Z_{i} \ar[rd]_{\tau_i} & & Z_{i+1} \ar[ld]\ar@{-->}[r]
	& ... \ar@{-->}[r] & Z_k\ar[d]\\ Y\ar[rrrd]_g & & & V_i\ar[d]_{\phi_i} & & & Y_1\ar[llld] \\ & & & U & & & }}\]
as in Theorem \ref{chf}.
\begin{lem}\label{ge1}
	Notation as above. Assume that $H_U\in|-K_U|$ is a reduced and irreducible effective Weil divisor such that
	$K_U+H_U$ is $\Q$-Cartier, $H_U$ has only Du Val singularities and contains the non-isomorphic locus of $Y\rightarrow U$. Then
	$H_{V_i}\in |-K_{V_i}|$ and $H_{V_i}\rightarrow H_U$ is a partial resolution. Moreover, we have $H_{V_0}\not\cong H_U$.
\end{lem}
\begin{proof}
	If $Y\rightarrow U$ is small, then $K_Y+H_Y=g\st(K_U+H_U)$. Assume that it is a divisorial contraction to a curve $C$.
	Since $H_U$ is normal, it is smooth at the generic point of $C$. Thus we also have $K_Y+H_Y=g\st(K_U+H_U)$.
	This implies that $H_Y\in|-K_Y|$, hence $H_Y$ passes through every
	non-Gorenstein singular point of $Y$. Corollary \ref{flopge} implies that $H_Y$ has also Du Val singularities. Thus we have
	\[K_{Z_0}+H_{Z_0}=h\st(K_Y+H_Y)=(g\circ h)\st(K_U+H_U)=(\phi_0\circ\tau_0)\st(K_U+H_U)\]
	by \cite[Lemma 2.7 (2)]{ch}. Since all $Z_i$ are isomorphic in codimension one, we have
	$K_{Z_i}+H_{Z_i}=(\phi_i\circ\tau_i)\st(K_U+H_U)$ for all $i$. Hence $K_{V_i}+H_{V_i}=\phi\st(K_U+H_U)$ for all $i$.
	This implies that $(V_i,H_{V_i})$ is canonical and so $H_{V_i}$ is normal by \cite[Proposition 5.51]{km} for all $i$.
	Hence $H_{V_i}$ has also Du Val singularities and $H_{V_i}\rightarrow H$ is a partial resolution.
	The fact that $H_{V_0}\not\cong H_U$ follows from \cite[Lemma 2.7 (3)]{ch}.
\end{proof}
Now let $W$ be a germ of three-dimensional non-$\Q$-factorial terminal singularity.
Let $W_1\rightarrow W$ be a $w$-morphism and let $Y\rightarrow W_1$ be a $\Q$-factorization. Let
\[ \vc{\xymatrix@C=0.6cm@R=1.5cm{ Y \ar@{}[r]|= \ar[d] & Y_0 \ar@{-->}[r] & ... \ar@{-->}[r] & Y_i\ar[rd] & 
& Y_{i+1}\ar[ld]\ar@{-->}[r] & ... \ar@{-->}[r] & Y_k\ar[d]\\ W_1\ar[d] & & & & U_i \ar[lllld] & & & X\ar[llllllld] \\ W & & & & & & }}\]
be a sequence of $K_Y$-MMP over $W$.
\begin{lem}\label{ge2}
	Assume that $H\in|-K_W|$ is a Du Val section. Then $H_{U_i}\rightarrow H$ is a partial resolution. 
\end{lem}
\begin{proof}
	By \cite[Lemma 2.7 (2)]{ch} and the fact that $Y\rightarrow W_1$ is isomorphic in codimension one, we know that $K_Y+H_Y=\phi\st(K_W+H)$,
	where $\phi$ is the morphism $Y\rightarrow W$. It follows that $K_{U_i}+H_{U_i}=\phi_i\st(K_W+H)$ where $\psi_i$ is the morphism
	$U_i\rightarrow W$, since $Y\dashrightarrow U_i$ is isomorphic is codimension one. Thus $H_{U_i}$ has only Du Val singularities
	and $H_{U_i}\rightarrow H$ is a partial resolution by Corollary \ref{flopge}.
\end{proof}

\subsection{Factorizing flops}
Let $X\dashrightarrow X'$ be a three-dimensional terminal flop over $W$. Let $W_1\rightarrow W$ be a strict $w$-morphism and let $Y$ be a
$\Q$-factorization of $W_1$. 
\begin{lem}
	$\rho(Y/W)=2$.
\end{lem}
\begin{proof}
	Let $X''$ be the relative minimal model of $Y\rightarrow W$. Let $E=exc(W_1\rightarrow W)$, then $E_Y$ is covered by
	$K_Y$-negative curves, so the MMP will contract $E_Y$. Thus $X''\rightarrow W$ is isomorphic in codimension one. This means that
	$X''$ is a $\Q$-factorization of $W$. By Remark \ref{minx} we know that $X''$ is isomorphic to either $X$ or $X'$, hence $\rho(X''/W)=1$.
	On the other hand, since there are only one exceptional divisor on $Y$ over $W$, there are only one divisorial contraction in the MMP.
	Hence $\rho(Y/X'')=1$ and so $\rho(Y/W)=2$. 
\end{proof}
If $W_1$ is $\Q$-factorial, then there are two $K_{W_1}$-negative extremal rays over $W$. One can run two different MMP and one may assume
that $X$ is one of the minimal model. We denote the other minimal model by $X_1$.
If $W$ is not $\Q$-factorial, then $Y\rightarrow W$ is a flopping contraction. One can construct a flop $Y\dashrightarrow Y'$ over $W_1$ as in
\cite{k}. We may assume that $X$ is a minimal model of $Y$ over $W$. Let $X_1$ be a minimal model of $Y'$ over $W$.
\begin{lem}
	Notation as above. One has $X_1=X'$.
\end{lem}
\begin{proof}
	Assume not, then $X_1=X$. We will show that this is impossible.\par
	If $W_1$ is not $\Q$-factorial, we denote \[Y=Y_0\dashrightarrow Y_1\dashrightarrow \cdots Y_k\rightarrow X\]
	be the $K_Y$-MMP over $W$, where $Y_i\dashrightarrow Y_{i+1}$ is a flip and $Y_k\rightarrow X$ is a divisorial contraction.
	Similarly we denote \[Y'=Y'_0\dashrightarrow Y'_1\dashrightarrow \cdots Y'_{k'}\rightarrow X_1\] 
	be the $K_{Y'}$-MMP over $W$. If $W_1$ is $\Q$-factorial, we denote
	\[ X\leftarrow Y_k\dashleftarrow\cdots\dashleftarrow Y_1\dashleftarrow Y_0=W_1
		=Y'_0\dashrightarrow Y'_1\dashrightarrow \cdots Y'_{k'}\rightarrow X_1\] 
	be the two different $K_{W_1}$-MMP over $W$. Assume that $X_1=X$, then $Y_k=Y'_{k'}$ because $Y_k\rightarrow X$ and
	$Y'_{k'}\rightarrow X_1=X$ extracts the same exceptional divisor. Interchanging $Y_j$ and $Y'_j$ if necessary we may assume that $k\geq k'$.
	Now $NE(Y_k/W)$ is generated by a $K_{Y_k}$-negative extremal ray and a $K_{Y_k}$-positive extremal ray. $Y_{k-1}$ and $Y'_{k'-1}$
	are both the anti-flip along the $K_{Y_k}$-positive extremal ray, hence we have $Y_{k-1}\cong Y'_{k'-1}$. By induction we may assume that
	$Y_{k-i}\cong Y_{k'-i}$ for all $i>0$, and so $Y_i\cong Y'$ for some $i$. Since $K_{Y'}$ is anti-nef over $W$ but $K_{Y_i}$ is not unless
	$i=0$, we know that $Y\cong Y'$ and both $Y\dashrightarrow Y_1$ and $Y'\dashrightarrow Y'_1$ contract the same extremal ray,
	this contradicts to our construction.
\end{proof}
In conclusion, we have
\begin{cor}\label{ff}
	Let $X\dashrightarrow X'$ be a three-dimensional terminal flop over $W$, then there exists a factorization
	\[ X\leftarrow Y_k\dashleftarrow\cdots\dashleftarrow Y_1\dashleftarrow Y_0=Y\dashrightarrow Y'=Y'_0
		\dashrightarrow Y'_1\dashrightarrow \cdots Y'_{k'}\rightarrow X_1\] 
	such that $Y_i\dashrightarrow Y_{i+1}$ and $Y'_{i'}\dashrightarrow Y'_{i'+1}$ are flips and either $Y=Y'$ or
	$Y\dashrightarrow Y'$ is a flop.
\end{cor}
The following lemma will be used in Section \ref{sce}.
\begin{lem}\label{gdfp}
	Let $X\dashrightarrow X'$ be a terminal flop over $W$. Then $dep(X)\leq dep(W)$. The equality holds
	if and only if $X\dashrightarrow X'$ is a Gorenstein flop.
\end{lem}
\begin{proof}
	If $X\dashrightarrow X'$ is a Gorenstein flop, then $dep(X)=dep(W)$. So we may assume that $X\dashrightarrow X'$ is a non-Gorenstein
	flop and we want to prove that $dep(X)<dep(W)$.\par
	We will prove the statement by induction on $gdep(W)$. We have a factorization
	as in Corollary \ref{ff}. If $gdep(W)=1$ then $W_1=Y=Y'$ is smooth. In this case $k=k'=0$ and $Y\rightarrow X$ is a smooth blow-down.
	This means that $X\dashrightarrow X'$ is a smooth flop, so $W$ is Gorenstein. Thus $dep(X)=dep(W)=0$.\par
	In general by the induction hypothesis we know that $dep(Y)\leq dep(W_1)$. If $k\neq 0$ then $dep(Y_k)<dep(Y)$ by Proposition \ref{ied}
	and so $dep(X)<dep(Y_k)+1\leq dep(Y)$. If $k=0$ then $Y_k\rightarrow X$ is a divisorial contraction to the flopping curve. One also has
	$dep(X)\leq dep(Y)$ by Proposition \ref{ied}. Thus \[ dep(X)\leq dep(Y)\leq dep(W_1)=dep(W)-1<dep(W).\]
\end{proof}
\begin{cov}\label{cn}
Let $Y^{(0)}_{(i)}=Y_i$, $U^{(0)}_{(i)}=U_i$. For any $n$-tuple $I=(a_1,...,a_n)$, we denote $I_n+1$ by the tuple $(a_1,...,a_n+1)$. Define 
\[\vc{\xymatrix@C=0.8cm{Y^{(j)}_{(I,0)}\ar[d]\ar@{-->}[r] & ...\ar@{-->}[r] & Y^{(j)}_{(I,i)} \ar[rd] & &
	Y^{(j)}_{(I,i+1)} \ar[ld]\ar@{-->}[r] & ... \ar@{-->}[r] & Y^{(j)}_{(I,k^{(j)}_{I})}\ar[d]\\
	Y^{(j)}_{I}\ar[rrrd] & & & U^{(j)}_{(I,i)}\ar[d] & & & Y^{(j)}_{I+1}\ar[llld] \\ & & & U^{(j)}_{I} & &	& }}\]
to be the factorization of the flip $Y^{(j)}_{I}\dashrightarrow Y^{(j)}_{I+1}$ as in Theorem \ref{chf}. Also, if
$Y^{(j)}_{(I,k^{(j)}_I)}\rightarrow Y^{(j)}_{I+1}$ is a divisorial contraction to a curve such that $Y^{(j)}_{(I,k^{(j)}_I)}$ is not Gorenstein
over $Y^{(j)}_{I+1}$, then we define $Y^{(j+1)}_I=Y^{(j)}_{(I,k^{(j)}_I)}$, $U^{(j+1)}_I=Y^{(j)}_{I+1}$ and
\[\vc{\xymatrix@C=0.8cm{Y^{(j+1)}_{(I,0)}\ar[d]\ar@{-->}[r] & ...\ar@{-->}[r] & Y^{(j+1)}_{(I,i)} \ar[rd] & &
	Y^{(j+1)}_{(I,i+1)} \ar[ld]\ar@{-->}[r] & ...\ar@{-->}[r] & Y^{(j+1)}_{(I,k^{(j+1)}_I)}\ar[d]\\
	Y^{(j+1)}_I\ar[rrrd] & & & U^{(j+1)}_{(I,i)}\ar[d] & & & Y^{(j+1)}_{I+1}\ar[llld] \\ & & & U^{(j+1)}_I & & & }}\]
to be the factorization of $Y^{(j)}_{(I,k^{(j)}_I)}\rightarrow Y^{(j)}_{I+1}$.
\end{cov}
We will keep this convention in the rest of this article. Nevertheless, usually every variety we are studying lies on the same
level, namely the superscript $(j)$ are all the same. In those situations we will omit the superscript and write $Y^{(j)}_I$ as $Y_I$.
Moreover, if $H$ is a divisor on some variety $Z$ which is birational to $Y^{(j)}_I$ and $H$ intersects the isomorphism locus of
$Z\dashrightarrow Y^{(j)}_I$ non-trivially. Then we will denote the proper transform of $H$ on $Y^{(j)}_I$ by $H^{(j)}_I$ or $H_I$ (in the
case that the superscript can be omit) instead of $H_{Y^{(j)}_I}$. 
\begin{lem}\label{sm}
	If $Y_{(I,0)}\dashrightarrow Y_{(I,1)}$ is a flop, then it is simple.
\end{lem}
\begin{proof}
	The flipping curve of $Y_I\dashrightarrow Y_{I+1}$ is a tree of $\Pp^1$'s. Let $C_1$, ..., $C_k$ be those flipping curves
	which passing though the non-isomorphic locus of $Y_{(I,0)}\rightarrow Y_I$. Then the proper transform of $C_i$ on
	$Y_{(I,0)}$ are all disjoint. Since flopping curves should be connected, the flop $Y_{(I,0)}\dashrightarrow Y_{(I,1)}$
	is simple.
\end{proof}

Combining Lemma \ref{ge1} and Lemma \ref{ge2}, one has 
\begin{lem}\label{ge}
	If $H\in|-K_W|$ is a Du Val section, then $H_{U^{(j)}_I}\in|-K_{U^{(j)}_I}|$ and
	$H_{U^{(j)}_{I}}\rightarrow H$ is a partial resolution for all possible $j$, $I$.
\end{lem}

\begin{lem}\label{gore}
	We have $\gd(X)\geq \gd(Y_I^{(j)})$ for all possible $j$, $I$. Moreover, if the exceptional locus of $X$ over $W$ contains a
	Gorenstein singular point, then the inequality is strict.
\end{lem}
\begin{proof}
	By Corollary \ref{feam} we know that $\gd(X)\geq \gd(Y_i)=\gd(Y^{(0)}_{(i)})$ for all $i$, and the inequality is strict if
	the exceptional locus of $X$ over $W$ contains a Gorenstein singular point. Now for any possible $I$ and $j$ one has that
	$\gd(Y_{I,l}^{(j)})\leq\gd(Y_{I+1}^{(j)})$ for all $0\leq l\leq k_I^{(j)}$ and also for $j>0$ one has
	$\gd(Y_{I+1}^{(j)})\leq\gd(U^{(j)}_I)=\gd(Y^{(j-1)}_I)$. This finishes the prove.
\end{proof}

\begin{pro}\label{fact1}
	Any three-dimensional terminal flop can be factorize into a composition of divisorial contractions to points, blow-up smooth curves,
	smooth flops, and inverses of above maps. 
\end{pro}
\begin{proof}
	In this proof we say that a birational map has a factorization if it can be factorize into a composition of
	(inverses of) divisorial contractions to points, blow-up smooth curves and smooth flops.\par
	Assume that $X\dashrightarrow X'$ is a singular flop over $W$. By Corollary \ref{ff} we have a decomposition
	$X\dashleftarrow Y\dashrightarrow Y'\dashrightarrow X'$. We need to show that $Y\dashrightarrow X$ has a factorization
	and $Y\dashrightarrow Y'$ has a factorization. Then by the symmetry we can also assume that $Y'\dashrightarrow X'$ has a factorization,
	so $X\dashrightarrow X'$ has a factorization.\par 
	As in Convention \ref{cn} we can a factorize the birational map
	$Y\dashrightarrow X$ to the following maps:
	\begin{enumerate}[(i)]
	\item $Y_I^{(j)}\leftarrow Y_{(I,0}^{(j)}$ which is an inverse of $w$-morphism.
	\item $Y_{(I,0)}^{(j)}\dashrightarrow Y_{(I,1)}^{(j)}$ which is a flop.
	\item $Y_{(I,k_I)}^{(j)}\rightarrow Y_{I+1}^{(j)}$ which is either a divisorial contraction to a point or a divisorial contraction
		to a curve such that $Y_{I,k_I}^{(j)}$ do not have non-Gorenstein point over $Y_{I+1}^{(j)}$. 
	\end{enumerate}
	Notice that if $Y_{I,k_I}^{(j)}\rightarrow Y_{I+1}^{(j)}$ is a divisorial contraction to a curve, then the curve is smooth by
	Lemma \ref{ge} and Corollary \ref{lci}. In this case by \cite[Theorem 4]{cu} we know that $Y_{I,k_I}^{(j)}\rightarrow Y_{I+1}^{(j)}$
	is a blowing-up a smooth curve. Thus $Y\dashrightarrow X$ has a factorization if whenever $Y^{(j)}_{I,0}\dashrightarrow Y^{(j)}_{I,1}$
	is a singular flop, it has a factorization. Let $\Ss_0$ be the set consists the flops $Y\dashrightarrow Y'$
	and $Y^{(j)}_{I,0}\dashrightarrow Y^{(j)}_{I,1}$ for all $j$ and $I$, and let
	$\Ss=\{Z\dashrightarrow Z'\mbox{ is a singular flop}\}\subset\Ss_0$.
	It is enough to show that $Z\dashrightarrow Z'$ has a factorization if $Z\dashrightarrow Z'\in\Ss$.\par
	Let $P\in W$ be the image of flopping curves of $X$ on $W$ and let $GE(P)=A_i$, $D_j$ or $E_k$ denotes the type of a general elephant near $P$.
	We will prove the statement by induction on the pair $(\gd(X),GE(P))$, where the relation between general elephants are given by
	\[ A_i<A_{i'}<D_j<D_{j'}<E_6<E_7<E_8\] if $i<i'$, $j<j'$. Notice that by Lemma \ref{gore} one has $\gd(Y^{(j)}_I)\leq\gd(X)$ for all
	possible $j$ and $I$. This means that for all $Z\dashrightarrow Z'\in\Ss$ one has that $\gd(Z)\leq\gd(X)$
	Moreover, if $X\dashrightarrow X'$ is a Gorenstein flop, then the inequality is strict. In this case by the induction hypothesis
	we know that the flop $Z\dashrightarrow Z'$ has a factorization for all $Z\dashrightarrow Z'\in\Ss$.\par
	Now assume that $X\dashrightarrow X'$ is a non-Gorenstein flop. Given $Z\dashrightarrow Z'\in\Ss$. If $Z\dashrightarrow Z'$ is
	a Gorenstein flop, then we already know that $Z\dashrightarrow Z'$ has a factorization. Assume that $Z\dashrightarrow Z'$ is a
	non-Gorenstein flop over $V$ and $Q\in V$ is the image of flopping curves, then $Q$ is a non-Gorenstein point. By Lemma \ref{ge}
	we know that $H_V\in|-K_V|$ and $H_V\rightarrow H$ is a partial resolution, for any Du Val section $H\in|-K_W|$. This means that
	$GE(Q)<GE(P)$, hence $Z\dashrightarrow Z'$ also has a factorization.
\end{proof}

\begin{proof}[Proof of Theorem \ref{fthm}]
	Assume that $X\dashrightarrow X'$ is a flop over $W$. By Proposition \ref{fact1} we already know that $X\dashrightarrow X'$
	can be factorize into a composition of (inverse of) divisorial contractions and smooth flops. We only need to show that
	we may further assume that those smooth flops are all simple. Notice that those flops $Y_{(I,0)}^{(j)}\dashrightarrow Y_{(I,1)}^{(j)}$
	are always simple by Lemma \ref{sm}, so we only need to consider the flop $Y\dashrightarrow Y'$.\par
	If $X\dashrightarrow X'$ is a smooth flop and $gdep(W)=1$, then $Y=W_1=Y'$ is smooth and so $Y_k=Y$. In this case the flop is simple
	(in fact it is an Atiyah flop). Now $Y\dashrightarrow Y'$ is a flop over $W_1$ and $gdep(W_1)=gdep(W)-1$, hence one can induction
	on $gdep(W)$ and assume that every smooth flops in the factorization of $Y\dashrightarrow Y'$ is simple. Hence every smooth flops
	in the factorization of $X\dashrightarrow X'$ is simple. 
\end{proof}

\section{Factorization of flips with simple singularities}\label{sflip}

We begin with the following important property of three-dimensional flips:
\begin{lem}[\cite{b}, Theorem 0]\label{bfc}
	Assume that $X\dashrightarrow X'$ is a three-dimensional canonical flip and $C$ is a flipping curve, then $-1<K_X.C<0$.
\end{lem}

\begin{lem}\label{cohF}
	Assume that $Y\dashrightarrow Y'$ is an Atiyah flop. Let $C_Y\subset Y$ be the flopping curve, $C_{Y'}\subset Y'$ be the flopped curve and
	$S\subset Y$ be a surface. Assume that either $S$ is smooth or $Sing(S)$ is pure of dimension one and $C_Y\not\subset Sing(S)$. Then
	either $S_{Y'}$ is smooth or $Sing(S_{Y'})$ is pure of dimension one.
\end{lem}
\begin{proof}
	Let $\pi:\tl{Y}=Bl_{C_Y}Y\rightarrow Y$ and $\pi':\tl{Y}=Bl_{C_{Y'}}Y'\rightarrow Y'$ be a common resolution of $Y\dashrightarrow Y'$.
	Let $E=exc(\pi)=exc(\pi')$. First assume that $S$ is smooth.
	\begin{enumerate}[(i)]
	\item Assume that $C_Y\not\subset S$. If $S.C_Y>1$ then $S_{Y'}$ is singular along $C_{Y'}$ by Lemma \ref{lati} and we have done.
		Assume that $S.C_Y=1$. We have that $S_{\tl{Y}}=Bl_{S\cap C_Y}S$ and $S_{Y'}\cong S_{\tl{Y}}$. Hence $S_{Y'}$ is smooth.
	\item Assume that $C_Y\subset S$ and $C_{Y'}\not\subset S_{Y'}$. In this case we have $S_{Y'}.C_{Y'}=1$ since otherwise $S$ won't be smooth.
		Now \[S\cong S_{\tl{Y}}=Bl_{S_{Y'}\cap C_{Y'}}S_{Y'}.\]
		Since $S.C_Y=-S_{Y'}.C_{Y'}=-1$, $C_Y$ is a $-1$-curve and $S_{Y'}$ is obtained by contracting this curve, so $S_{Y'}$ is also smooth.
	\item Assume that $C_Y\subset S$ and $C_{Y'}\subset S_{Y'}$. We have $\pi\st S=S_{\tl{Y}}+E$.
		Assume that $S_{Y'}$ is not singular along $C_{Y'}$, then we also have ${\pi'}\st S_{Y'}=S_{\tl{Y}}+E=\pi\st S$. This implies that
		$S.C_Y=S_{Y'}.C_{Y'}=0$. Let $\phi:Y\rightarrow U$ and $\phi':Y'\rightarrow U$ be the flopping contractions. Since $S$ is
		smooth, $S_U$ has only a $A_1$ singularity. Thus $S_{Y'}$ is smooth because $S_{Y'}\not\cong S_U$ and there are no intermediate varieties
		between a $A_1$ singularity and its minimal resolution.
	\end{enumerate}\par
	Now assume that $S$ is not smooth, $Sing(S)$ is pure of dimension one and $C_Y\not\subset Sing(S)$. Let $Sing(S)_{Y'}$ be the proper
	transform of $Sing(S)$ on $Y'$. We want to say that $Sing(S_{Y'})$ is either $Sing(S)_{Y'}$ or $C_{Y'}\cup Sing(S)_{Y'}$.
	For any $P\in Sing(S)\cap C_Y$, we have $\mmul_P S\geq 2$. Let $l_P=\pi\st P$. Since $C_Y\not\subset Sing(S)$,
	$\pi\st S=S_{\tl{Y}}+\lambda E$ for $\lambda=0$ or $1$. Since $\mmul_{l_P}E=1$, we have $l_P\subset S_{\tl{Y}}$. We may write
	\[S_{\tl{Y}}\cap E=\sum_{P\in Sing(S)\cap C_Y}m_Pl_P+\lambda\Gamma\] for some curve $\Gamma$ which maps bijectively to $C_Y$ via $\pi$.
	If $\sum_Pm_P>1$ or $\pi'(\Gamma)=C_{Y'}$ then $S_{Y'}$ is singular along $C_{Y'}$, so $Sing(S_{Y'})=C_{Y'}\cup Sing(S)_{Y'}$.
	Now assume that $\sum_Pm_P=1$ and $\pi'(\Gamma)$ is a point. Notice that in this case $\lambda=1$ or $P$ can not be a singular point of $S$.
	We have $Sing(S_{Y'})\cap C_{Y'}=\pi'(\Gamma)$. It follows that $Sing(S)_{Y'}\cap C_{Y'}=\pi'(\Gamma)$ and so $Sing(S_{Y'})=Sing(S)_{Y'}$.
\end{proof}

\begin{cov}\label{flipd}
	Let $Y_I\dashrightarrow Y_{I+1}$ be a flip. We say that the flip has a factorization $(\ast)$ if we have the following diagram
	\[\vc{\xymatrix{ & & Y_{\bar{I}} \ar@{-->}[r] \ar@{}[d]|{\vdots} & Y_{\bar{I}+1} \ar@{}[d]|{\vdots} \\
		& Y_{(I,1,0)} \ar@{-->}[r]\ar[d] & Y_{(I,1,1)} & Y_{(I,1,2)}\ar[d] \\
		Y_{(I,0)} \ar@{-->}[r]\ar[d] & Y_{(I,1)} & & Y_{(I,2)} \ar[d] \\
		Y_I & & & Y_{I+1} }}\] with $\bar{I}=(I,1,...,1,0)$, such that
	\[Y_{\bar{I}+1}\rightarrow\cdots Y_{(I,1,2)}\rightarrow Y_{(I,2)}\rightarrow Y_{I+1}\]
	is a sequence of $w$-morphisms or blowing-up smooth curves.
\end{cov}

\begin{lem}\label{ati1}
	Let $Y_I\dashrightarrow Y_{I+1}$ be a flip. Assume that $Y_I\dashrightarrow Y_{I+1}$ has a factorization $(\ast)$ as in
	Convention \ref{flipd} and every flop in the factorization is an Atiyah flops. If $S$ is a
	surface on $Y_{I+1}$ such that $S_{I'}$ is smooth in the smooth locus of $Y_{I'}$ for $I'=(I,1,..,1,2)$ or $\bar{I}$,
	then either $S_I$ has one-dimensional singularities, or $S_I$ is smooth in the smooth locus of $Y_I$.
\end{lem}
\begin{proof}
	We know that $Y_{\bar{I}}\dashrightarrow Y_I$ can be factorize into a composition of Atiyah flops and $w$-morphism.
	Notice that the singular locus of $S_{(I,1,...,1,1)}$ do not contain the flopped curve of 
	$Y_{(I,1,...,1,0)}\dashrightarrow Y_{(I,1,...,1,1)}$ since the flopped curve appears on the smooth locus of $Y_{(I,1,...,1,2)}$.
	Lemma \ref{cohF} implies that either $S_{(I,1,...,1,0)}$ has one-dimensional singularities or $S_{(I,1,...,1,0)}$ is smooth
	in the smooth locus of $Y_{(I,1,...,1,0)}$. One can see that $S_{Y_I}$ satisfies the desired property.
\end{proof}
\begin{lem}\label{ati2}
	Let $Y_I\dashrightarrow Y_{I+1}$ be a flip such that $Y_{(I,0)}\dashrightarrow Y_{(I,1)}$ is a simple smooth flop.
	Let $F_{(I,0)}$ be the exceptional divisor of $Y_{(I,0)}\rightarrow Y_I$. If $F_{(I,1)}$ is smooth along the
	flopped curve of $Y_{(I,0)}\dashrightarrow Y_{(I,1)}$, then $Y_{(I,0)}\dashrightarrow Y_{(I,1)}$ is an Atiyah flop.
\end{lem}
\begin{proof}
	Let $C_{(I,1)}$ be the flopped curve on $Y_{(I,1)}$. We have $(C_{(I,1)})^2_{F_{(I,1)}}<0$ by the
	Hodge index theorem \cite[Lemma 3.40]{km}. Since $F_{(I,0)}.C_{(I,0)}>0$ where $C_{(I,0)}$ is the flopping curve, we have
	$F_{(I,1)}.C_{(I,1)}<0$. So 
	\[ 0>F_{(I,1)}.C_{(I,1)}=K_{F_{(I,1)}}.C_{(I,1)}
		=-2-(C_{(I,1)})^2_{F_{(I,1)}}>-2,\] which implies that
	$F_{(I,1)}.C_{(I,1)}=-1$. Hence the normal bundle of $Y_{(I,1)}$ along $C_{(I,1)}$ is
	$\Oo(-1)\oplus\Oo(-1)$ (cf. \cite[Remark 5.2]{re}) and so $Y_{(I,0)}\dashrightarrow Y_{(I,1)}$ is an Atiyah flop.
\end{proof}
\begin{lem}\label{ati}
	Let $Y_I\dashrightarrow Y_{I+1}$ be a flip such that $Y_{(I,0)}\dashrightarrow Y_{(I,1)}$ is a smooth flop.
	Assume that both $Y_I$ and $Y_{I+1}$ has only $cA/r$ singularities of the form
	\[ (xy+z^r+f(z,u)=0)\subset\A^4_{(x,y,z,u)}/\frac{1}{r}(\alpha,-\alpha,1,r),\] $Y_I\dashrightarrow Y_{I+1}$ has a factorization
	$(\ast)$ as in Convention \ref{flipd} and $Y_{(I,I',0)}\dashrightarrow Y_{(I,I',1)}$
	is either an isomorphism or an Atiyah flop for all possible non-empty $I'=(1,...,1)$.  
	Then $Y_{(I,0)}\dashrightarrow Y_{(I,1)}$ is an Atiyah flop.
\end{lem}
\begin{proof}
	Let $F_{(I,2)}$ be the exceptional divisor of $Y_{(I,2)}\rightarrow Y_{I+1}$. By the assumption that $Y_{I+1}$
	has only simple $cA/r$ singularities one can verify that $F_{(I,2)}\subset Y_{(I,2)}$ satisfies the condition in Lemma \ref{ati1}
	(applied to the flip $Y_{(I,1)}\dashrightarrow Y_{(I,2)}$). Hence 
	either $F_{(I,1)}$ has one-dimensional singularities, or it is smooth in the smooth locus of $Y_{(I,1)}$. However it can not
	has one-dimensional singularities since every curve on $Y_{(I,1)}$ appears either on $Y_{(I,0)}$ or $Y_{(I,2)}$, and
	$F_{(I,i)}$ is smooth in the smooth locus of $Y_{(I,i)}$ for $i=0$, $2$ by the assumption about the singularities of
	$Y_I$ and $Y_{I+1}$. Hence $F_{(I,1)}$ is smooth in the smooth locus of $Y_{(I,1)}$ and now
	$Y_{(I,0)}\dashrightarrow Y_{(I,1)}$ is an Atiyah flop by Lemma \ref{ati2}.
\end{proof}

\begin{lem}\label{lflip}
	Assume that $dep(Y_I)=dep(Y_{I+1})+1$. Then $Y_{(I,0)}\dashrightarrow Y_{(I,1)}$ is a flop.
	Either $Y_{(I,1)}\rightarrow Y_{I+1}$ is a divisorial contraction to a curve which do not change the depth,
	or $Y_{(I,1)}\dashrightarrow Y_{(I,2)}$ is a flip with $dep(Y_{(I,1)})=dep(Y_{(I,2)})+1$ and
	$Y_{(I,2)}\rightarrow Y_{I+1}$ is a $w$-morphism.
\end{lem}
\begin{proof}
	Let $C_{(I,0)}$ be the flipping/flopping curve on $Y_{(I,0)}$ and let $F_{(I,0)}$ be the exceptional
	divisor of $Y_{(I,0)}\rightarrow Y_I$. Then we have $F_{(I,0)}.C_{(I,0)}>0$, hence $F_{(I,1)}.C_{(I,1)}<0$,
	where $C_{(I,1)}$ denotes the flipped/flopped curve on $Y_{(I,1)}$. This implies that $C_{(I,1)}\subset F_{(I,1)}$.\par
	Assume that $Y_{(I,0)}\dashrightarrow Y_{(I,1)}$ is a flip. Then by Proposition \ref{ied} we have 
	\[dep(Y_{(I,1)})\leq dep(Y_{(I,0)})-1=dep(Y_I)-2=dep(Y_{I+1})-1.\]
	This implies that $k_I=1$ and $Y_{(I,1)}\rightarrow Y_{I+1}$ is a $w$-morphism by Proposition \ref{ied}. Since $F_{(I,1)}$ is contracted
	to a point on $Y_{I+1}$, one can see that $K_{Y_{(I,1)}}.C_{(I,1)}<0$, this leads to a contradiction. Hence
	$Y_{(I,0)}\dashrightarrow Y_{(I,1)}$ is a flop.\par
	If $k_I=1$, then $Y_{(I,1)}\rightarrow Y_{I+1}$ is a divisorial contraction to a curve because we have
	$exp(Y_{(I,1)}\rightarrow U_I)=F_{Y_{(I,1)}}$. One can see that $dep(Y_{(I,1)})=dep(Y_I)-1=dep(Y_{I+1})$. If $k_I>1$, then $k_I$ should be
	$2$ and we have $dep(Y_{(I,2)})=dep(Y_{I+1})-1$. This proves the last statement.
\end{proof}

\begin{lem}\label{flopness}
	Assume that $Y_I\dashrightarrow Y_{I+1}$ is a flip such that the flipping curve passes through only one singular point. If
	$Y_{(I,0)}\dashrightarrow Y_{(I,1)}$ is a flop, then it is a Gorenstein flop. 
\end{lem}
\begin{proof}
	Assume that then the flopping curve $C_{(I,0)}$ passes through a non-Gorenstein point. Let $H_I\in|-K_{Y_I}|$ be a
	general elephant of the flipping contraction $Y_I\rightarrow U_I$, then $H_{(I,0)}\in |-K_{Y_{(I,0)}}|$ by \cite[Lemma 2.7 (2)]{ch}.
	Since $C_{(I,0)}$ contains a non-Gorenstein point, $H_{(I,0)}$ intersects $C_{(I,0)}$ non-trivially. Since $H_{(I,0)}.C_{(I,0)}=0$,
	$H_{(I,0)}$ contains $C_{(I,0)}$. This implies that $H_I$ contains $C_I$ where $C_I$ is the image of $C_{(I,0)}$ on $Y_I$. Notice that
	$C_I$ is a flipping curve of $Y_I\dashrightarrow Y_{I+1}$ by Remark \ref{flpc}. Since $C_I$ passes through only one singular point,
	the flip $Y_I\dashrightarrow Y_{I+1}$ is of type $IC$ in Koll\'{a}r-Mori's classification
	\cite[Theoren 2.2]{km2}. It follows that $Y_I$ has only a cyclic quotient singular point over $U_I$, and $Y_{(I,0)}$ also has only a cyclic
	quotient singular point over $U_{(I,0)}$ by Remark \ref{car}. According to \cite[Theorem 2.2.2]{km2}, the Dynkin diagram of the
	$H_{U_I}$ is of type $D_m$ 
	\[ \vc{\xymatrix@R=0.2cm@C=0.2cm{ & & & \circ \ar@{-}[d] & 
		\\ \circ \ar@{-}[r] & \cdots \ar@{-}[r] & \circ \ar@{-}[r] & \circ \ar@{-}[r] & \bullet } }.\] 
	Let $P_{(I,0)}$ be the non-$\Q$-factorial point of $U_{(I,0)}$. It follows that near $P_{(I,0)}$ we have $H_{U_{(I,0)}}$ is of type
	$A_i$ for $i=3$ or $m-1$, or $D_j$ for some $4\leq j\leq m-1$. We will prove that any of those cases can not happen.
	\begin{enumerate}[(i)]
	\item $H_{U_{(I,0)}}$ is of type $A_i$ near $P_{(I,0)}$. Then the flopping curve of $Y_{(I,0)}\dashrightarrow Y_{(I,1)}$
		contains a cyclic quotient point of index $i$, so $P_{(I,0)}$ is of index $i$. One can easily see that this can not happen because
		$P_{(I,0)}$ can not be a cyclic quotient point, and the general elephant of a non-cyclic-quotient index $i$ point has at least
		$A_{2i+1}$ singularities (\cite[(6.4)]{re2}).
	\item $H_{U_{(I,0)}}$ is of type $D_j$ near $P_{(I,0)}$. In this case the flopping curve of $Y_{(I,0)}\dashrightarrow Y_{(I,1)}$
		contains a cyclic quotient point of index $j$ and hence $P_{(I,0)}$ is also an index $j$ point. This means that $j=4$ and 
		$P_{(I,0)}$ is a $cAx/4$ point. However, the general elephant of a $cAx/4$ point has at least $D_5$ singularities (\cite[(6.4)]{re2}).
		Hence this case won't happen.
	\end{enumerate}	 
\end{proof}

\begin{lem}\label{pago}
	Assume that $Y_I\dashrightarrow Y_{I+1}$ is a flip such that the flipping curve passes through only one singular point $P$ with
	$\gd(P\in X)=0$ and $dep(P\in X)=1$. Then $Y_{(I,0)}\dashrightarrow Y_{(I,1)}$ is an Atiyah flop and
	$Y_{(I,1)}\rightarrow Y_{I+1}$ is a blowing-up a smooth curve.\par
	Moreover, we have $K_{Y_{I+1}}.C_{I+1}=1$ if $C_{I+1}$ is the flipped curve.
\end{lem}
\begin{proof}
	Note that $gdep(P\in X)=1$, hence $Y_{(I,0)}$ is smooth. Since there are no flipping contraction begin with a smooth variety, we have that 
	$Y_{(I,0)}\dashrightarrow Y_{(I,1)}$ is a flop and $Y_{(I,1)}\rightarrow Y_{I+1}$ is a blowing-up a smooth curve
	by \cite[Remark 3.4]{ch}. This flop is an Atiyah flop by Lemma \ref{ati}.\par
	Now we know that $F_{(I,1)}.C_{(I,1)}=-1$ by the computation in Lemma \ref{ati2}, hence $K_{Y_{I+1}}.C_{I+1}=1$.
\end{proof}
\begin{lem}\label{3flip}
	Assume that $Y_I\dashrightarrow Y_{I+1}$ is a flip such that the flipping curve passes through only one singular point $P$ with
	$\gd(P\in X)=0$ and $dep(P\in X)=2$. Then we have one of the following factorization.
	\begin{enumerate}[(1)]
	\item \[\vc{\xymatrix{ & Y_{(I,1,0)} \ar@{-->}[r] \ar[d]_w & Y_{(I,1,1)} \ar[d]^c \\
			 Y_{(I,0)} \ar@{-->}[r] \ar[d]_w  & Y_{(I,1)} & Y_{(I,2)} \ar[d]^w \\ Y_I & & Y_{I+1}}}.\]
		$Y_{I+1}$ has a $\frac{1}{2}(1,1,1)$ singularity. 
	\item \[\vc{\xymatrix{ Y_{(I,0,0)} \ar@{-->}[r] \ar[d]_w & Y_{(I,0,1)} \ar[d]^c \\
			 Y_{(I,0)} \ar[d]_w  & Y_{(I,1)} \ar[d]^c \\ Y_I & Y_{I+1}}}.\]
		$Y_{I+1}$ is smooth and $a(exc(Y_{(I,0,1)}\rightarrow Y_{(I,1)}),Y_{I+1})=2$. 
	\item \[\vc{\xymatrix{ & Y_{(I,1,0)} \ar@{-->}[r] \ar[d]_w & Y_{(I,1,1)} \ar[d]^c \\
			 Y_{(I,0)} \ar@{-->}[r] \ar[d]_w  & Y_{(I,1)} & Y_{(I,2)} \ar[d] \\ Y_I & & Y_{I+1}}}.\]
		$Y_{I+1}$ is smooth and $Y_{(I,2)}\rightarrow Y_{I+1}$ is a blowing-up a smooth curve or a point. 
	\end{enumerate}
	Here every dash map is an Atiyah flop.\par
	Moreover, if the singular point on $Y_I$ is of type $cA/2$, then we are in the case (1) or (3).
\end{lem}
\begin{proof}
	Notice that $\gd(Y_{I,0})=0$ by Remark \ref{sww}. Assume first that $Y_{(I,0)}\dashrightarrow Y_{(I,1)}$ is a flip. We have
	$dep(Y_{(I,1)})=0$ and so $k_I=1$. The factorization of $Y_{(I,0)}\dashrightarrow Y_{(I,1)}$ is given by Lemma \ref{pago}.
	Notice that $exc(Y_{(I,1)}\rightarrow Y_{I+1})$ intersects the flipped curve of $Y_{(I,0)}\dashrightarrow Y_{(I,1)}$ negatively, hence it
	contains the flipped curve. This implies that $Y_{(I,1)}\rightarrow Y_{I+1}$ should be a divisorial contraction to a curve. It is easy to see
	that $a(exc(Y_{(I,0,1)}\rightarrow Y_{(I,1)}),Y_{I+1})=2$ and we are in the case (2).\par
	Now assume that $Y_{(I,0)}\dashrightarrow Y_{(I,1)}$ is a flop. The flop is a smooth flop by Lemma \ref{flopness} and the fact that
	$\gd(Y_{I,0})=0$. If $k_I=1$ then $Y_{(I,1)}\rightarrow Y_{I+1}$ is a divisorial contraction to a curve (cf. \cite[Remark 3.4]{ch}).
	However we know that $\gd(Y_{(I,1)})=\gd(Y_{(I,0)})=0$ and $dep(Y_{(I,1)})=dep(Y_{(I,0)})=1$. By Remark \ref{sdep} we know that
	$Y_{(I,1)}$ has a $\frac{1}{2}(1,1,1)$ singularity, so $Y_{I+1}$ can not be smooth. Since $dep(Y_{I+1})<2$, we know that $dep(Y_{I+1})=1$
	and $Y_{I+1}$ has a $\frac{1}{2}(1,1,1)$ singularity. However, there is no divisorial contraction to a curve which contains a
	cyclic quotient singularity \cite[Theorem 5]{ka2}. Hence $k_I>1$ and we have a flip $Y_{(I,1)}\dashrightarrow Y_{(I,2)}$.
	Now $dep(Y_{(I,2)})=0$ so $k_I=2$ and we have a divisorial contraction $Y_{(I,2)}\rightarrow Y_{I+1}$. The factorization of
	$Y_{(I,1)}\dashrightarrow Y_{(I,2)}$ is given by Lemma \ref{pago}. If $Y_{(I,2)}\rightarrow Y_{I+1}$ is a $w$-morphism, then $dep(Y_{I+1})=1$
	and we we are in the case (1). Otherwise $Y_{I+1}$ is smooth and we are in the case (3).\par
	Now every flop appear in the factorization is an Atiyah flop by Lemma \ref{ati}. If the singular point on $Y_I$ is of type $cA/2$,
	then $K_{Y_I}.C_{Y_I}=-\frac{1}{2}$ by Lemma \ref{bfc}, where $C_{Y_I}$ is the flipping curve of $Y_I\dashrightarrow Y_{I+1}$. 
	An easy computation shows that $K_{Y_{(I,0)}}.C_{(I,0)}>-\frac{1}{2}$. On the other hand, we know that $Y_{(I,0)}$ has only a
	$\frac{1}{2}(1,1,1)$ singularity, so $K_{Y_{(I,0)}}.C_{(I,0)}\in\frac{1}{2}\Z_{\leq 0}$. This means that
	$K_{Y_{(I,0)}}.C_{Y_{(I,0)}}=0$ and so $Y_{(I,0)}\dashrightarrow Y_{(I,1)}$ is a flop. Hence we are in the case (1) or (3).
\end{proof}
\section{Factorizing $D$-type simple smooth flops}\label{scd}

The gaol of this section is to construct the factorization of a $D$-type simple smooth flop.

\begin{lem}
	Assume that \[W=(x^2+y^2z+z^n+ug(x,y,z,u)=0)\subset \A^4\] is an isolated $cD$ singularity. After a suitable coordinate change we may write
	the defining equation of $W$ as $x^2+y^2z+\lambda yu^k+g(z,u)$. We say that this equation is a normal form of the singularity of $W$.
\end{lem}
\begin{proof}
	After several steps of coordinate change one may assume that $W$ is defined by $f(x,y,z,u)=x^2+y^2z+\lambda yu^k+\mu u^n+y^2uh(y,u)+g(z,u)$.
	We use the notation that $k=\infty$ if $\lambda=0$ and $n=\infty$ if $\mu=0$. Let $n(f)=\min\{n,k\}$. Then $n(f)$ is finite or
	$W$ do not have isolated singularities.
	We may assume that for all monomial $y^iu^j\in h(y,u)$ we have $j+1<n(f)$. We call such kind of equation a sub-normal form.\par
	In the sub-normal form we may assume that $h(y,u)$ contains only finitely many monomials. 
	If $h(y,u)=\sum_{i,j}a_{ij}y^iu^j$, we define $h(y,u)_{\deg y>0}=\sum_{i>0,j}a_{ij}y^iu^j$. Let 
	\[\delta(f)=\left\lbrace \begin{tabular}{ll} $\min\se{j}{y^iu^j\in h(y,u),i\neq0}$ & if $h(y,u)_{\deg y>0}\neq 0$, \\
		$n(f)-1$ & otherwise.\end{tabular}\right.\]
	If $\delta(f)=n(f)-1$, then after replace $z$ by $z-uh(u)$ the equation becomes a normal form and we have done. Now assume that
	$\delta(f)<n(f)-1$. Let $z'=z+uh(y,u)$ and $f'(x,y,z',u)$ be a sub-normal form of $f(x,y,z',u)$. One may write
	$f'(x,y,z',u)=x^2+y^2z'+\lambda yu^k+\mu u^n+y^2uh'(y,u)+g'(z',u)$, where
	\[ y^2uh'(y,u)=\sum_{i\neq0,j} b_{ij}(-uh(y,u)_{\deg_y>0})^iu^j-
		\left(\sum_{i\neq0,j} b_{ij}(-uh(y,u)_{\deg_y>0})^iu^j\right)_{\deg y\leq 1;\deg_u\geq n(f')}\]
	if $g(z,u)=\sum_{i,j}b_{ij}z^iu^j$. It follows that $n(f')\leq n(f)$ and $\delta(f')\geq \delta(f)$ if $\delta(f')\neq n(f')-1$. Thus
	by induction on the integer $n(f)-\delta(f)$ one can show that every isolated $cD$ singularity admits a normal form.
\end{proof}

\begin{lem}\label{dres}
	Assume that $W$ has a $cD_4$ singularity such that there is only one exceptional divisor of discrepancy one over $W$. Let $W_1$ be the
	unique $w$-morphism over $W$. Then $W_1$ has a non-Gorenstein singularity of type $\frac{1}{2}(1,1,1)$. The other singularities of
	$W_1$ may be $cD_4$, $cA_1$ or $cA_2$. Furthermore, assume that there exists a non-$\Q$-factorial Gorenstein point $P$ of type $cD_4$ on
	$W_1$, then there is only one exceptional divisor of discrepancy one over $P$.
\end{lem}
\begin{proof}
	We may assume that $W$ is defined by $f(x,y,z,u)=x^2+y^2z+\lambda yu^k+g(z,u)$. Let $f_3(y,z,u)$ be the homogeneous part of degree 3 of $f$.
	As discussed in \cite[Proposition 15]{c}, we have that a
	$w$-morphism over $W$ can be obtained by the weighted blow-up with following weights:
	\begin{enumerate}[(i)]
	\item $f_3(y,z,u)$ is irreducible. One can define $W_1$ to be the weighted blow-up with weight $(2,1,1,1)$.
		The only non-Gorenstein singularity of $W_1$ is the origin of the chart $U_x$, which is a $\frac{1}{2}(1,1,1)$ singularity.
	\item $k>2$ and $f_3(y,z,u)$ is reducible. After a suitable change of coordinates we may assume that $z^3\in g(z,u)$.
		One can weighted blow-up $(2,1,2,1)$. In this case there is a $cA/2$ singularity on the origin of the chart $U_z$,
		which is locally defined by \[(x_1^2+y_1^2+z_1^2+u_1^{2k}=0)\subset \A^4_{(x_1,y_1,z_1,u_1)}/\frac{1}{2}(0,1,1,1),\]
		for some $k>0$. Let $F$ be the exceptional divisor which is obtained by weighted blowing-up $W_1$ with weight $w$ such that
		$w(y_1+z_1,y_1-z_1,u_1,x_1)=\frac{1}{2}(1,k-1,k,2)$. Then $F$ is also an exceptional divisor of discrepancy one over $W$,
		which contradicts to our assumption. Thus this case won't happen.
	\item $k=2$ and $f_3(y,z,u)$ is reducible. Let $W_1$ be the weighted-blow-up with weight $(2,2,1,1)$. Then there is a non-Gorenstein
		singularity on $W_1$ which is locally defined by \[(x_2^2+y_2z_2+u_2^2=0)\subset \A^4_{(x_2,y_2,z_2,u_2)}/\frac{1}{2}(2,1,1,1).\]
		Let $F$ be the exceptional divisor obtained by weighted blow-up $\frac{1}{2}(2,1,1,1)$. One can compute that $a(W,F)=1$,
		hence this case won't happen.
	\end{enumerate}
	From now on we assume that $f_3(y,z,u)$ is irreducible and $W_1$ is obtained by weighted blowing-up $(2,1,1,1)$. One can check that
	every Gorenstein singular point on $W_1$ is contained in $U_z$. Fix a singular point $P\in W_1$. We may assume that $P$ is the origin
	of the chart $U_z$ and the local defining equation is $y^2+x^2z+g'(z,u)$ with $\mul g'(z,u)\leq 3$. Hence $P$ is a $cA_1$, a $cA_2$ or a
	$cD_4$ point.\par
	Assume that $P$ is a $cD_4$ point. Note that in this case we have $u^3\in f_3(y,z,u)$ and hence the degree $3$ part of $y^2+x^2z+g'(z,u)$
	is irreducible. Thus there are only one discrepancy one exceptional divisor over $P$ by the discussion above.
\end{proof}

Because $W_1$ may have $cA_2$ singularities, smooth flops over $cA_2$ points may appear in the factorization of a type $D$ simple smooth flop.
Luckily, the factorization of a smooth flop over a $cA_2$ point is similar to the factorization of a $D$-type simple smooth flop. We will
construct the factorization of this two kinds of flops at the same time.

\begin{pro}\label{cd}
	Notation as in Convention \ref{cn}. Assume that $X\dashrightarrow X'$ is a smooth flop over $W$ such that either
	\begin{enumerate}[(i)]
	\item $X\dashrightarrow X'$ is simple and $W$ has a $cD$ singularity, or
	\item $W$ has a $cA_2$ singularity.
	\end{enumerate}
	Then $Y_{(0,1)}\rightarrow Y_{(1)}\rightarrow X$ is a sequence of blowing-up smooth curves,
	$Y_{(0,0)}\dashrightarrow Y_{(0,1)}$ is an Atiyah flop, and $Y_{(0,0)}\rightarrow Y_{(0)}$
	is a blowing-up a $\frac{1}{2}(1,1,1)$ point.\par
	On the other hand, if $Y\not\cong Y'$, then $Y\dashrightarrow Y'$ is a flop which is either an $A$-type simple smooth flop or
	satisfies one of the condition (i), (ii) above.
\end{pro}
\begin{proof}
	By Lemma \ref{dres} in case (i) or by direct computation in case (ii), we know that $Y_{(0)}$ has only one non-Gorenstein singular point,
	which is of type $\frac{1}{2}(1,1,1)$. Notice that this point is the only singular point of $Y_{(0)}$ because
	$\gd(X)=0$ implies that $\gd(Y_{(0)})=0$ by Corollary \ref{feam}. The factorization now follows from Lemma \ref{pago}.
	The last statement follows from Lemma \ref{dres} (2).
\end{proof}

\section{Factorizing $E$-type simple smooth flop}\label{sce}
\subsection{Resolving $cE$ singularities}
Let $W=(x^2+y^3+yg\gq3(z,u)+h\gq4(z,u)=0)\subset\A^4$ be a $cE$ singularity. We have
\begin{enumerate}
	\item $h_4(z,u)\neq0$ if $X$ has a $cE_6$ singularity.
	\item $g_3(z,u)\neq0$ and $h_4(z,u)=0$ if $W$ has a $cE_7$ singularity.
	\item $g_3(z,u)=h_4(z,u)=0$ and $h_5(z,u)\neq0$ if $W$ has a $cE_8$ singularities.
\end{enumerate}
We may assume that $u^4\not\in h(z,u)$ (resp. $u^5\not\in h(z,u)$) if $W$ is $cE_6$ (resp. $cE_8$).
\begin{lem}\label{cEw}
	Let $w_k(x,y,z,u)=(a,b,c,1)$ be a weight such that after weighted blowing-up $W$ with weight $w_k$ we get a $w$-morphism.
	Here $k=a+b+c-1$ equals to the weight of the defining equation of $X$.
	Assume that there is only one exceptional divisor with discrepancy one over $X$, then
	\begin{enumerate}[(1)]
	\item If $2a=3b=k$, then $(a,b,k)=(3,2,6)$.
	\item If $k>3$, then neither $2a-1=3b=k$ nor $3b-1=2a=k$.
	\end{enumerate}	 
\end{lem}
\begin{proof}
	Let $W'\rightarrow W$ be the weighted blow-up with weight $w$ and let $E$ be the exceptional divisor. We have $a(E,W)=1$. We want to
	say that if (1) or (2) is not true, then there exists an exceptional divisor $F$ over $W$ such that $a(F,W)=1$, which contradicts
	to our assumption.\par 
	Assume first that (1) is not ture, so that $2a=3b=k$ but $k>6$. In this case we have $a=3d$ and $b=2d$ for some $d>1$. Consider the chart
	$U'_y\subset W'$ is defined by $x^2+1+g'(y,z,u)=0\subset \A^4/\frac{1}{b}(a,-1,c,1)$.
	The point $(\pm1,0,0,0)$ is a cyclic quotient point of type $\frac{1}{d}(-1,c,1)$. Notice that $c=d+1$ by the relation
	\[ 6d=2a=3b=k=a+b+c-1=5d+c-1. \]
	Let $W''\rightarrow W'$ be the weighted blow-up $\frac{1}{d}(d-1,1,1)$. One can take $F=exc(W''\rightarrow W)$.\par
	Now assume that (2) is not true. If $3b-1=2a=k$, then the origin of the chart $U'_y$ on $W'$ is a cyclic quotient point of type
	$\frac{1}{b}(a,c,1)$. Note that $a+c=k+1-b=3b-b=2b$ and $a>b$ since $3b-1=2a$. Let $W''\rightarrow W'$ be the weighted blow-up
	with weight $\frac{1}{b}(a-b,c,1)$ and we can take $F=exc(W''\rightarrow W)$.\par
	Finally assume that $2a-1=3b=k$. The origin of the chart $U'_x$ on $W'$ is a cyclic quotient point of type $\frac{1}{a}(b,c,1)$.
	Notice that we have $b+c=k+1-a=a$ and $2b=\frac{2}{3}(2a-1)>a$ because $k>3$ implies $a>2$.
	Let $W''\rightarrow W'$ be the weighted blow-up with weight $\frac{1}{a}(2b-a,2c,2)$ and let $F$ be the exceptional divisor.
	One can compute that \[ a(F,W)\leq \frac{1}{a}\left(3(2b-a)+2\right)=\frac{1}{a}\left(2(2a-1)-3a+2\right)=1,\]
	hence $a(W,F)=1$.
\end{proof}
Assume that $W$ is a $cE_n$ singularity for $n=6$, $7$ or $8$. We assume that $W$ is defined by $x^2+y^3+yg\gq3(z,u)+h\gq4(z,u)$.
By \cite[Theorem 34, 36, 37]{c} we know that there exists a weight $w_k$ in the following set
	\[\left\lbrace\begin{tabular}{c}
		$w_4=(2,2,1,1)_{n=6}$, $w_5=(3,2,1,1)_{n=6,7}$, $w_6=(3,2,2,1)$, $w_8=(4,3,2,1)$,\\
		$w_9=(5,3,2,1)_{n=7,8}$, $w_{12}=(6,4,3,1)$, $w_{14}=(7,5,3,1)_{n=7,8}$,\\
		$w_{18}=(9,6,4,1)_{n=7,8}$, $w_{24}=(12,8,5,1)_{n=8}$, $w_{30}=(15,10,6,1)_{n=8}$\end{tabular}\right\rbrace.\]
such that the weighted blow-up with weight $w_k$ gives a $w$-morphism $W'\rightarrow W$.
\begin{lem}\label{cew}
	Notation as above. Assume that $W$ has only one exceptional divisor with has discrepancy one, then
	$W'\rightarrow W$ is obtained by weighted blowing-up $w_{n-2}$. Moreover, if $n=6$ then $h_4(z,u)$ is not a perfect square. If $n=7$,
	then either $u^3\in g(z,u)$ or $u^5\in h(z,u)$.
\end{lem}
\begin{proof}
	By Lemma \ref{cEw} we know that if $W'\rightarrow W$ is obtained by weighted blowing-up the weight $w_k$, then $k\leq 6$. This proves the
	case when $n=8$.\par
	Let $E=exc(W'\rightarrow W)$. Assume that $n=7$. We want to show that if $W'\rightarrow W$ is obtained by weighted blowing-up $w_6$,
	then there exists an exceptional divisor $F\neq E$ such that $a(F,W)=1$. We know that $W$ is defined by
	$x^2+y^3+yz^3+yg\gq3(z,u)+h\gq5(z,u)$. The chart $U'_z\subset W'$ is defined by
	\[(x^2+y^3+yz^2+yg'\gq2(z,u)+h'\gq4(z,u)=0)\subset\A^4_{(x,y,z,u)}/\frac{1}{2}(1,0,1,1),\]
	which is a $cD/2$ point. 
	\begin{enumerate}[(i)]
	\item $\mul h'(z,u)\geq6$. Let $W''\rightarrow W'$ be the weighted blow-up with weight $\frac{1}{2}(3,4,1,1)$ and let
		$F=exc(W''\rightarrow W)$.
	\item $\mul h'(z,u)=4$ and $u^4\not\in h'(z,u)$. Note that after suitable change of coordinate we may assume that $u^2\not\in g'(z,u)$.
		Let $W_1\rightarrow W'$ be the weighted blow-up with weight $\frac{1}{2}(3,2,3,1)$. The chart $U_{1,z}\subset W_1$ is defined by
		\[ x^2+y^3+yz+g_1(y,z,u)=0\subset \A^4_{(x,y,z,u)}/\frac{1}{3}(3,2,1,1).\]
		Let $W_2\rightarrow W_1$ be the weighted blow-up $\frac{1}{3}(3,2,1,1)$ (resp. $\frac{1}{3}(3,5,1,1)$) if $u^3\in g_1(y,z,u)$
		(resp. $u^3\not\in g_1(y,z,u)$). One can see that either $F_1=exc(W_1\rightarrow W')$ or $F_2=exc(W_2\rightarrow W_1)$
		has discrepancy one over $W$. Let $F$ be this exceptional divisor.
	\item $u^4\in h'(z,u)$. Let $W''\rightarrow W'$ be the weighted blow-up with weight $\frac{1}{2}(3,2,1,1)$ and let $F=exc(W''\rightarrow W)$.
	\end{enumerate}
	Since there always exists an exceptional divisor $F$ over $W$ such that $a(F,W)=1$, we get a contradiction. Thus when $n=7$, $W'\rightarrow W$
	should be obtained by weighted blowing-up $w_5$. It is easy to see that $w_5$ defines a $w$-morphism implies either $u^3\in g(z,u)$ or
	$u^5\in h(z,u)$.\par
	Finally assume that $n=6$ and we are going to show that $W'\rightarrow W$ is obtained by weighted blowing-up $w_4$. Equivalently,
	we want to say that if $W'\rightarrow W$ is obtained by weighted blowing-up $w_6$ or $w_5$, then there exists a exceptional divisor $F$ over
	$W$ such that $a(F,W)=1$.
	\begin{enumerate}[(a)]
	\item $W'\rightarrow W$ is obtained by weighted blowing-up $w_6$. We have that $W$ is defined by $x^2+y^3+z^4+g(y,z,u)$.
		The chart $U'_z\subset W'$ is defined by \[(x^2+y^3+z^2+g'(y,z,u)=0)\subset \A^4_{(x,y,z,u)}/\frac{1}{2}(1,0,1,1),\]
		which is a $cA/2$ point. For a suitable change of coordinates we may write the defining equation as $(x+iz)(x-iz)+g_1(y,u)$.
		Define a weight $w'(x+iz,x-iz,u,y)=(1,2m-1,1,2)$, where $m=\w_{w'}(g_1(y,u))$. Let $W''\rightarrow W'$ be the weighted blow-up with
		the weight $w'$, then $F=exc(W''\rightarrow W')$ satisfies that $a(F,W)=1$.
	\item $W'\rightarrow W$ is obtained by weighted blowing-up $w_5$. In this case $W$ is defined by \[x^2+xq(z,u)+y^3+g(y,z,u)\]
		and the weight of the defining equation of $W$ is $5$. Note that we have $q(z,u)$ is a non-zero homogeneous
		polynomial of degree $2$. The chart $U'_x\subset W'$ is a cyclic quotient point of type $\frac{1}{3}(2,1,1)$. Let
		$W''\rightarrow W'$ be the weighted blow-up with weight $\frac{1}{3}(2,1,1)$, then $F=exc(W''\rightarrow W')$ satisfies $a(F,W)=1$.
	\end{enumerate}
	Hence $W'\rightarrow W$ is defined by $w_4$. One can easily see that $w_4$ is a $w$-morphism only when $h_4(z,u)$ is not a perfect square. 
\end{proof}

\subsection{Factorizing $cE_6$ flops}
In this subsection we assume that $X\dashrightarrow X'$ is a simple smooth flop over $W$, such that $W$ has a $cE_6$ singularity.
As before we let $W_1\rightarrow W$ be a $w$-morphism, let $Y\rightarrow W_1$ be a $\Q$-factorization of $W_1$ and construct
a diagram as in Convention \ref{cn}
\begin{lem}\label{lce6}$ $
	\begin{enumerate}[(1)]
	\item $W_1$ has only one singular point, which is a depth three $cAx/2$ point.
	\item $Y\dashrightarrow Y'$ is a simple flop over the $cAx/2$ point. $Y$ has a $cA/2$ point which is of depth $2$. There are no other
		singular point on $Y$.
	\end{enumerate}
\end{lem}
\begin{proof}
	By Lemma \ref{cew} we know that $W_1\rightarrow W$ is obtained by weighted blowing-up the weight $(2,2,1,1)$. One can check that there
	is only one non-Gorenstein point $P\in W_1$, which is a $cAx/2$ point defined by
	\[(x^2+y^2+h'(y,z,u)=0)\subset\A^4_{(x,y,z,u)}/\frac{1}{2}(0,1,1,1)\]
	such that $\mul h'(y,z,u)=4$ and $h'(0,z,u)$ is not a perfect square (the latter statement follows from Lemma \ref{cew}).
	The only $w$-morphism $W_2\rightarrow W_1$ over this point is given by weighted blowing-up the weight $\frac{1}{2}(2,3,1,1)$
	(\cite[Section 8]{h1}). An easy computation shows that the only non-Gorenstein point on $W_2$ is a $\frac{1}{3}(2,1,1)$ point,
	hence $dep(W_2)=2$ and $dep(W_1)=3$. Assume that $P$ is $\Q$-factorial, then the only non-Gorenstein point
	on $Y$ is also a $cAx/2$ point. Hence the flipping curve of $Y=Y_{(0)}\dashrightarrow Y_{(1)}$ should pass through this
	$cAx/2$ point and there are no other non-Gorenstein points on the flipping curve. However it is impossible because of the classification
	\cite[Theorem 2.2]{km2}. Hence $P$ is not $\Q$-factorial and so $Y\dashrightarrow Y'$ is a flop over $P$.
	Note that $Y$ has no Gorenstein singular points since $0\leq \gd(Y)\leq \gd(X)=0$. Hence $W_1$ do not have any Gorenstein singular point
	other than $P$.\par
	\begin{cl}
		$dep(Y)>1$.
	\end{cl}
	To prove the claim, assume that $dep(Y)=1$. We have that $Y\dashrightarrow Y_{(1)}$ can be factorize as in Lemma \ref{pago}. Hence there
	exists a sequence of blowing-up smooth curves $Y_{(0,1)}\rightarrow Y_{(1)}\rightarrow X$, so that there exists an Atiyah flop
	$Y_{(0,1)}\dashrightarrow Y_{(0,0)}$. However since $W$ has a $cE_6$ singularity, the normal bundle sequence is of length $4$, this leads a
	contradiction.\par
	Since $dep(Y)<dep(W_1)=3$ by Lemma \ref{gdfp}, we have $dep(Y)=2$. Since $Y\rightarrow W_1$ is a fopping contraction over a $cAx/2$ point,
	the singular point of $Y$ is a $cA/2$ point described in Remark \ref{sdep}.
\end{proof}
\begin{lem}\label{cax2}
	The flop $Y\dashrightarrow Y'$ in Lemma \ref{lce6} can be factorize into
	\[ Y\leftarrow Z_{(1)}\dashleftarrow Z_{(0)}\dashrightarrow Z'_{(0)} \dashrightarrow Z'_{(1)} \rightarrow Y',\]
	such that $Z_{(0)}\dashrightarrow Z_{(1)}$ and $Z'_{(0)}\dashrightarrow Z'_{(1)}$ are flips in Lemma
	\ref{3flip} (1), $Z_{(0)}\dashrightarrow Z'_{(0)}$ is at worst a $D$-type smooth flop, and other arrows are $w$-morphisms.
\end{lem}
\begin{proof}
	Let $W_2\rightarrow W_1$ be a $w$-morphism over the $cAx/2$ point. Form the computation in Lemma \ref{cax} we know that
	the non-Gorenstein point on $W_2$ is a $\frac{1}{3}(2,1,1)$ point and $W_2$ has at worst $cD$ Gorenstein singularities. Hence
	$dep(Z_{(0)})=2$ and $Z_{(0)}\dashrightarrow Z'_{(0)}$ is at worst a $D$-type flop. We know that $dep(Y)=2$ and $Y\dashrightarrow Y'$
	is a non-Gorenstein flop by Lemma \ref{cax}. It follows that the birational map $Z_{(0)}\dashrightarrow Y$ is a flip
	$Z_{(0)}\dashrightarrow Z_{(1)}$ followed by a $w$-morphism $Z_{(1)}\rightarrow Y$. 
	One has that $dep(Z_{(1)})=1$ and $Z_{(0)}\dashrightarrow Z_{(1)}$ is given by Lemma \ref{3flip} (1). This proves the lemma.
\end{proof}

\begin{pro}\label{ce6}$ $
	Assume that $X\dashrightarrow X'$ is a simple smooth flop over $W$ such that $W$ has a $cE_6$ singularity. Then we have a factorization
	\[X\leftarrow Y_{(1)}\dashleftarrow Y_{(0)}=Y\dashrightarrow Y'=Y'_{(0)}\dashrightarrow Y'_{(1)} \rightarrow X'\]
	such that $Y_{(0)}\dashrightarrow Y_{(1)}$ is given by Lemma \ref{3flip} (3), $Y\dashrightarrow Y'$
	is given by Lemma \ref{cax2}, and all other maps are blowing-up smooth curves. The diagram of $Y'\dashrightarrow X'$ is symmetric.
\end{pro}
\begin{proof}
	By Lemma \ref{lce6} we know that $dep(Y)=2$. Since $Y$ has $cA/2$ singularities and $Y_{(1)}$ is smooth, $Y_{(0)}\dashrightarrow Y_{(1)}$
	is given by Lemma \ref{3flip} (3).
\end{proof}
\subsection{Factorizing $cE_7$ flops} 
In this subsection we assume that $X\dashrightarrow X'$ is a simple smooth flop over $W$, such that $W$ has a $cE_7$ singularity.
\begin{lem}\label{lce7}
	$Y\dashrightarrow Y'$ is a smooth flop over a $cA_2$ point. There are two singular points on $Y$. One of them is a $\frac{1}{3}(2,1,1)$
	point and the other one is a $\frac{1}{2}(1,1,1)$ point. Moreover, the flipping curve of $Y=Y_{(0)}\dashrightarrow Y_{(1)}$ passes through
	the both singular points.
\end{lem}
\begin{proof}
	Locally $W$ is defined by $x^2+y^3+yz^3+g(y,z,u)$. By Lemma \ref{cew} we know that $W_1\rightarrow W$ is obtained by weighted blowing-up
	the weight $(3,2,1,1)$. It is easy to see that there are two non-Gorenstein points, one of them is a $\frac{1}{3}(2,1,1)$ point and the
	other one is a $\frac{1}{2}(1,1,1)$ point. Note that the chart $U_y\subset W_1$ is defined by 
	\[(y(x^2+1)+z^3+g'(y,z,u)=0)\subset\A^4_{(x,y,z,u)}/\frac{1}{2}(1,0,1,1).\] One can see that the point $P=(\pm1,0,0,0)$ is a $cA_2$ point.
	Since $\gd(Y)\leq \gd(X)=0$, $P$ is not $\Q$-factorial and $Y\dashrightarrow Y'$ should be a smooth flop over $P$.\par
	Let $H_W=(u=0)\subset W$ be a Du Val section. Then $H_{W_1}$ passes through the non-$\Q$-factorial point of $W_1$, hence
	$H_Y$ contains flopping curves of $Y\dashrightarrow Y'$. Let $C_Y$ be the flipping curve of $Y=Y_{(0)}\dashrightarrow Y_{(1)}$.
	Then $C_Y$ intersects a flopping curve non-trivially, hence $C_Y$ intersects $H_Y$ non-trivially at a smooth point. 
	Now there is an index $I=(0,...,0)$ and a sequence of $w$-morphisms \[Y_I\rightarrow ...\rightarrow Y_{(0,0)}\rightarrow Y_{(0)}=Y\]
	such that $K_{Y_I}.C_{Y_I}=0$, hence $H_{Y_I}.C_{Y_I}=0$. We also know that $H_{Y_I}$ intersects $C_{Y_I}$ non-trivially, hence
	$C_{Y_I}\subset H_{Y_I}$ which means that $C_Y\subset H_Y$. Now $C_{W_1}\subset H_{W_1}.E$ and $H_{W_1}.E=(z=u=0)\subset\Pp(3,2,1,1)$
	is irreducible where $E=exc(W_1\rightarrow W)$. Hence $C_{W_1}=H_{W_1}.E$ and one can see that $C_{W_1}$ contains the two non-Gorenstein
	points. Hence $C_Y$ contains the two non-Gorenstein points.
\end{proof}
\begin{lem}\label{a2flop}
	Let $Y$ be a terminal threefold contains two $\frac{1}{2}(1,1,1)$ points. Assume that $Y\rightarrow U$ is a birational map contracting
	$C\cong\Pp^1$ to a point such that $-K_Y$ is nef over $U$ and the two $\frac{1}{2}(1,1,1)$ points are contained in $C$. Then $K_Y.C=0$.
	Moreover, assume that there exists a flop $Y\dashrightarrow Y'$ over $U$ such that the general elephant $H_U\in|-K_U|$ has $A$-type
	singularities, then we have the following factorization
	\[ Y\leftarrow Z_{(1)}\dashleftarrow Z_{(0)}=Z'_{(0)}\dashrightarrow Z'_{(1)}\rightarrow Y',\]
	where $Z_{(1)}\rightarrow Y$ as well as $Z'_{(1)}\rightarrow Y'$ are $w$-morphisms and $Z_{(0)}\dashrightarrow Z_{(1)}$ as well as
	$Z'_{(0)}\dashrightarrow Z'_{(1)}$ has a factorization given in Lemma \ref{3flip} (1).
\end{lem}
\begin{proof}
 	Let $P_1$ and $P_2$ be the two cyclic quotient points. We have that $Y\rightarrow U$ can not be a flipping contraction because
 	$\sum_{i=1,2}w_{P_i}(0)\geq 1$ (please see \cite[(2.3), Theorem 4.9]{mo}). Hence it is a flopping contraction.\par
	Now assume that we have a flop $Y\dashrightarrow Y'$ over $Q\in U$ and assume that a general elephant near $Q$ has $A$-type singularities.
	Then $Q$ is a $cA/2$ point. There are exactly two discrepancy $\frac{1}{2}$ exceptional divisors and one discrepancy $1$ exceptional
	divisor over $Q\in U$. By \cite[Proposition 3.4]{me} we may assume that $Q$ is defined by
	\[ (xy+z^4+u^2+g(z,u)=0)\subset \A^4_{(x,y,z,u)}/\frac{1}{2}(1,1,1,0).\]
 	Let $Z=Z'$ be the weighted blow-up with the weight $\frac{1}{2}(3,1,1,2)$, then the only singular point on $Z$ is a $\frac{1}{3}(1,1,2)$
 	point. Hence $dep(Z)=2$.\par
 	On the other hand, $Z_{(k)}\rightarrow Y$ is a $w$-morphism, so $dep(Z_{(k)})=dep(Y)-1=1$. This implies that $k=1$. The factorization
 	of $Z_{(0)}\dashrightarrow Z_{(1)}$ is given by Lemma \ref{3flip} (1).
\end{proof}
\begin{pro}\label{ce7}$ $
	Assume that $X\dashrightarrow X'$ is a simple smooth flop over $W$ such that $W$ has a $cE_7$ singularity.
	Then there exists a factorization
	\[  X\leftarrow \bar{Y}\leftarrow \tl{Y}\dashleftarrow Y_{(0,2)} \dashleftarrow Y_{(0,1)} \dashleftarrow Y_{(0,0)}
	\rightarrow Y_{(0)}=Y \dashrightarrow Y'=Y'_{(0)}\leftarrow Y'_{(0,0)} \]
	\[\dashrightarrow Y'_{(0,1)}\dashrightarrow Y'_{(0,2)}\dashrightarrow \tl{Y'}\rightarrow \bar{Y'}\rightarrow X'\]
	such that $\tl{Y}\rightarrow\bar{Y}\rightarrow X$ is a sequence of blowing-up smooth curves,
	$Y_{(0,1)}\dashrightarrow Y_{(0,2)}$ and $Y_{(0,2)}\dashrightarrow \tl{Y}$ are given in Lemma \ref{pago},
	$Y_{(0,0)}\dashrightarrow Y_{(0,1)}$ is a flop given in Lemma \ref{a2flop} and $Y\dashrightarrow Y'$ is a smooth flop over a $cA_2$ point.
	The diagram of $Y'\dashrightarrow X'$ is symmetric.
\end{pro}
\begin{proof}
	By Lemma \ref{lce7} there are two cyclic quotient points on $Y_{(0)}$ which have of indices $2$ and $3$ respectively and the flopping curve
	$C_Y$ of $Y=Y_{(0)}\dashrightarrow Y_{(1)}$ passes through the both singular point. By the classification \cite[Theorem 2.2]{km2} we know
	that we are in the case semistable $IA+IA$, hence a general elephant $H_W\in|-K_W|$ has $A$-type singularities.\par
	Now $Y_{(0,0)}$ contains two $\frac{1}{2}(1,1,1)$ point. Since $H_W$ has $A$-type singularities, $H_{U_{(0,0)}}$ has $A$-type singularities.
	Thus by Lemma \ref{a2flop} we know that $Y_{(0,0)}\dashrightarrow Y_{(0,1)}$ is a flop and its factorization is given by Lemma \ref{a2flop}.
	\begin{cl}
		We have $k_{(0)}\geq 2$. There exists a flip $Y_{(0,2)}\dashrightarrow \tl{Y}$ and a sequence of blowing-up smooth curves
		$\tl{Y}\rightarrow \bar{Y}\rightarrow X$. 
	\end{cl}
	To prove the claim, assume first that $k_{(0)}=1$. In this case we have a divisorial contraction to a curve $Y_{(0,1)}\rightarrow Y_{(1)}$
	(\cite[Remark 3.4]{ch}). Let $\Gamma\subset Y_{(1)}$ be this curve. Then there are two singular fibers over $\Gamma$ because there are
	two singular points on $Y_{(0,1)}$ which are connected by the flopped curve of the flop $Y_{(0,0)}\dashrightarrow Y_{(0,1)}$, and the
	flopped curve maps bijectively to $\Gamma$. On the other hand, we have $dep(Y_{(1)})<dep(Y_{(0)})=3$, hence $dep(Y_{(1)})\leq 2$. This
	implies that $Y_{(1)}$ contains two depth $1$ points, which are $\frac{1}{2}(1,1,1)$ points. However there is no divisorial contraction
	to a curve which contains cyclic quotient points by \cite[Theorem 5]{ka2}. This proves that $k_{(0)}\geq 2$.\par
	Now assume that $k_{(0)}=2$. There is a flip followed by a divisorial contraction
	$Y_{(0,1)}\dashrightarrow Y_{(0,2)}\rightarrow Y_{(1)}$. Notice that if $dep(Y_{(1)})=1$ then the $K_Y$-MMP over $W$ is given by
	\[ Y=Y_{(0)}\dashrightarrow Y_{(1)}\dashrightarrow Y_{(2)}\rightarrow X\]
	such that $Y_{(1)}\dashrightarrow Y_{(2)}$ is given by Lemma \ref{pago} and $Y_{(2)}\rightarrow X$ is a blowing-up a smooth curve on $X$.
	As shown in the proof of Lemma \ref{lce6}, in this case the normal bundle sequence of the flop $X\dashrightarrow X'$ has only length three.
	This leads to a contradiction. Hence $dep(Y_{(1)})\neq 1$ and we have $dep(Y_{(1)})$ is either $0$ or $2$. Assume that $dep(Y_{(1)})=0$,
	then $Y_{(1)}$ is smooth. Hence $Y_{(0,2)}$ is also smooth and $Y_{(0,2)}\rightarrow Y_{(1)}$ is divisorial contraction to a curve because
	the exceptional divisor of this divisorial contraction has discrepancy two over $X$, hence discrepancy one over $Y_{(1)}$. This implies
	that the two singular points of $Y_{(0,1)}$ are both contained in the flipping curve of $Y_{(0,1)}\dashrightarrow Y_{(0,2)}$. However it is
	impossible since the both two singular points are $\frac{1}{2}(1,1,1)$ points and there are no flip along two $\frac{1}{2}(1,1,1)$ points
	by Lemma \ref{a2flop}. So $dep(Y_{(1)})$ should be $2$.\par
	In this case we have $dep(Y_{(0,2)})=1$ and $Y_{(0,2)}\rightarrow Y_{(1)}$ is a $w$-morphism. In other words, we have
	$Y_{(0,2)}=Y_{(1,0)}$. Let $\tl{Y}=Y_{(1,1)}$ and $\bar{Y}=Y_{(2)}$. We know that $Y_{(2)}$ can not has depth one or the length of the
	normal bundle sequence should be $3$. Thus $Y_{(2)}$ is smooth and so $\bar{Y}=Y_{(2)}\rightarrow X$ is a blowing-up a smooth curve. Since
	$dep(Y_{(0,2)})=dep(Y_{(1,0)})=1$, $Y_{(1,1)}$ is also smooth and $Y_{(1,1)}\rightarrow Y_{(2)}$ is also a blowing-up a smooth curve.\par
	Finally if $k_{(0)}>2$ then it is easy to see that $k=3$ and $dep(Y_{(0,i)})=3-i$ for $i=2$, $3$. As we saw before, $dep(Y_{(1)})$
	can not be one. Hence $dep(Y_{(1)})=0$. This implies that there exists a divisorial contraction $Y_{(1)}\rightarrow X$
	which is a blowing-up a smooth curve. Let $F=exc(Y_{(0,3)}\rightarrow Y_{(1)})$. One can check that $a(F,X)=2$, hence
	$a(F,Y_{(1)})=1$ and so $Y_{(0,3)}\rightarrow Y_{(1)}$ is also a blowing-up a smooth curve. We let
	$\tl{Y}=Y_{(0,3)}$ and $\bar{Y}=Y_{(1)}$ and the claim is proved.\par
	Notice that $Y_{(0,i)}$ has only $\frac{1}{2}(1,1,1)$ singularities for $i=1$, $2$. Hence the factorization of the two flips
	$Y_{(0,1)}\dashrightarrow Y_{(0,2)}$ and $Y_{(0,2)}\dashrightarrow \tl{Y}$ are given by Lemma \ref{pago}.
	$Y\dashrightarrow Y'$ is a smooth flop over a $cA_2$ point by Lemma \ref{lce7}.
\end{proof}
\subsection{Factorizing $cE_8$ flops}
In this subsection we assume that $X\dashrightarrow X'$ is a simple smooth flop over $W$, such that $W$ has a $cE_8$ singularity.
\begin{lem}\label{le8}
	There exists a sequence of $w$-morphisms $W_8\rightarrow\cdots W_1\rightarrow W_0=W$ such that $W_8$ is smooth. $W_1$ has a
	non-$\Q$-factorial $cE/2$ singular point, $W_2$ has a non-$\Q$-factorial $cD/3$ singular point, $W_3$ has a $cAx/4$ singular point,
	$W_4$ has a $\frac{1}{5}(3,2,1)$ point, and
	$W_5$ has a $\frac{1}{3}(1,2,1)$ and a $\frac{1}{2}(1,1,1)$ point. $W_6$ and $W_7$ is obtained by resolving the $\frac{1}{3}(1,3,1)$
	singularity and $W_8$ is the blowing-up the $\frac{1}{2}(1,1,1)$ point. Moreover, let $E_i=exc(W_i\rightarrow W_{i-1})$, then one can
	compute $a(E_j,W_i)$ as in the following table:
	\[\vc{\begin{tabular}{l|ccccccccc}
	 & $W$ & $W_1$ & $W_2$ & $W_3$ & $W_4$ & $W_5$ & $W_6$ & $W_7$ \\\hline
	 $E_1$ & $1$ & - & - & - & - & - & - & - \\
	 $E_2$ & $2$ & $\frac{1}{2}$ & - & - & - & - & - & - \\
	 $E_3$ & $3$ & $1$ & $\frac{1}{3}$ & - & - & - & - & - \\
	 $E_4$ & $4$ & $\frac{3}{2}$ & $\frac{2}{3}$ & $\frac{1}{4}$ & - & - & - & - \\
	 $E_5$ & $5$ & $2$ & $1$ & $\frac{1}{2}$ & $\frac{1}{5}$ & - & - & - \\
	 $E_6$ & $2$ & $1$ & $\frac{2}{3}$ & $\frac{1}{2}$ & $\frac{2}{5}$ & $\frac{1}{3}$ & - & - \\
	 $E_7$ & $4$ & $2$ & $\frac{4}{3}$ & $1$ & $\frac{4}{5}$ & $\frac{2}{3}$ & $\frac{1}{2}$ & - \\
	 $E_8$ & $3$ & $\frac{3}{2}$ & $1$ & $\frac{3}{4}$ & $\frac{3}{5}$ & $\frac{1}{2}$ & $\frac{1}{2}$ & $\frac{1}{2}$ \\
	\end{tabular}}\]
\end{lem}
\begin{proof}
	We may assume that $W$ is defined by $x^2+y^3+z^5+g(y,z,u)$ such that $\frac{\partial^2}{\partial y^2}g(y,z,u)=0$ and $u^5\not\in g(y,z,u)$.
	By Lemma \ref{cew}, we know that $W_1\rightarrow W$ is obtained by weighted blowing-up $(3,2,2,1)$. The only non-Gorenstein singular point
	on $W_1$ is a $cE/2$ point. If this point is $\Q$-factorial, then the only non-Gorenstein point on $Y$ is this $cE/2$ point, so the
	flipping curve of $Y=Y_0\dashrightarrow Y_1$ should pass through this point. However it is impossible because of the classification
	\cite[Theorem 2.2]{km2}. So the $cE/2$ point is not $\Q$-factorial. In this case $W_1$ do not have any Gorenstein singular point since
	$\gd(Y)\leq \gd(X)=0$. In particular, we have that at least one of the following monomials \[yu^4,zu^4,u^6,u^7\] appears in $g(y,z,u)$.\par
	Now the chart $U_{z,1}\subset W_1$ is defined by \[ (x^2+y^3+z^4+g'(y,z,u)=0)\subset \A^4_{(x,y,z,u)}/\frac{1}{2}(1,0,1,1)\]
	and at least one of the following following monomials \[ yu^4,u^4,u^6,zu^7\] appears in $g'(y,z,u)$. Assume that $u^4\in g'(y,z,u)$,
	we choose a suitable change of coordinates $z\mapsto z+\lambda u$ so that $u^4\not\in g'(y,z+\lambda u,u)$, and construct the weighted
	blow-up with weight $\frac{1}{2}(3,2,3,1)$. Under this construction we get a $w$-morphism $W'\rightarrow W_1$ and one can see that the
	exceptional divisor of $W'\rightarrow W_1$ has discrepancy one over $W$. This leads a contradiction since there should be only one
	discrepancy one exceptional divisor over $W$. Hence $u^4\not\in g'(y,z,u)$.\par
	Now let $W_2\rightarrow W_1$ be the weighted blow-up with weight $\frac{1}{2}(3,2,3,1)$. The only non-Gorenstein singular point
	on $W_2$ is the origin of the chart $U_{z,2}\subset W_2$ which is defined by
	\[ (x^2+y^3+z^3+g''(y,z,u)=0)\subset \A^4_{(x,y,z,u)}/\frac{1}{3}(0,2,1,1),\] such that $yu^4$, $u^6$ or $z^2u^7\in g''(y,z,u)$.
	Let $w(y,z,u)=\frac{1}{3}(2,4,1)$ be a weight. The monomials with $w$-weights less than twelve are the following:
	\begin{enumerate}[(a)]
	\item $w$-weight equals to $6$: $yu^4$, $u^6$ and $zu^2$.
	\item $w$-weight equals to $9$: $yzu^3$, $yu^7$, $z^2u$, $zu^5$ and $u^9$. Notice that
		$yu^7$ and $u^9$ can not appear in $g''(y,z,u)$ because $g''(y,z,u)=g'(yz,z^{\frac{3}{2}},z^{\frac{1}{2}}u)/z^3$. 
	\end{enumerate}
	One can see that
	\begin{enumerate}[(i)]
	\item If $\w_wg''(y,z,u)<12$, then the weighted blow-up with weight $\frac{1}{3}(3,2,4,1)$ defines a $w$-morphism $W_3\rightarrow W_2$.
	\item If $\w_wg''(y,z,u)\geq 12$, then the weighted blow-up with weight $\frac{1}{3}(6,5,4,1)$ defines a $w$-morphism $W_3\rightarrow W_2$.
	\end{enumerate}
	Notice that there are exactly one exceptional divisor which has discrepancy $\frac{1}{2}$ over $W_1$ because any exceptional divisor 
	which do not appear on $W_2$ has discrepancy greater than $\frac{1}{2}\cdot\frac{4}{3}=\frac{2}{3}$ over $W_1$. Hence
	the $Y$ has only one non-Gorenstein point.
	\begin{cl}
		The $cD/3$ point on $W_2$ is not $\Q$-factorial.
	\end{cl}	
	To prove the claim, consider the following three possibilities:
	\begin{description}
		\item[Case (i-1)] $zu^2\in g''(y,z,u)$. We will show that this case won't happen. In this case the chart $U_{z,3}\subset W_3$
			is defined by \[ (x^2+y^3+z^2+g'''(z,u)=0)\subset \A^4_{(x,y,z,u)}/\frac{1}{4}(3,2,1,1),\] with $u^2\in g'''(y,z,u)$.
			After a suitable change of coordinates $z\mapsto z+\lambda u$ we may assume that $U_{z,3}$ is defined by $x^2+y^3+zu$. Let
			$W_4\rightarrow W_3$ be the weighted blow-up with weight $\frac{1}{4}(3,2,5,1)$. One can verify that the exceptional divisor of
			$W_4\rightarrow W_3$ has discrepancy one over $W$. This leads to a contradiction.
		\item[Case (i-2)] $zu^2\not\in g''(y,z,u)$. The origin of the chart $U_{z,3}$ is a $cAx/4$ point and weighted blow-up this point with
			weight $\frac{1}{4}(3,2,5,1)$ defines a $w$-morphism $W_4\rightarrow W_3$. $W_4$ has a $\frac{1}{5}(3,2,1)$ point. Let $F$ be the
			exceptional divisor which has discrepancy $\frac{3}{5}$ over $W_4$. Then one can compute that $a(F,W_2)=1$ and
			$a(F,W_1)=\frac{3}{2}$. Lemma \ref{nqf} now implies the claim.
		\item[Case (ii)] The only non-Gorenstein singular point on $W_3$ is a cyclic quotient point of type $\frac{1}{5}(1,4,1)$. Let $F$ be the
			exceptional divisor which has discrepancy $\frac{4}{5}$ over $W_3$. We have $a(F,W_2)=1$ and $a(F,W_1)=\frac{3}{2}$. The statement
			follows from Lemma \ref{nqf}.
	\end{description}
	Now the claim implies that $W_2$ do not have any Gorenstein singularity because the $\Q$-factorization of $W_2$ has zero Gorenstein depth.
	This implies that $zu^7$ do not appear in $g'(y,z,u)$. Hence either $yu^4$ or $u^6$ appear in $g''(y,z,u)$
	and weighted blowing-up the $cD/3$ point on $W_2$ with the weight $\frac{1}{3}(3,2,4,1)$ defines a $w$-morphism $W_3\rightarrow W_2$.\par
	The only non-Gorenstein point on $W_3$ is the origin of the chart $U_{z,3}$ defined by
	 \[ (x^2+y^3+z^2+g'''(z,u)=0)\subset \A^4_{(x,y,z,u)}/\frac{1}{4}(3,2,1,1)\] and the resolution of this point is described in Case (i-2)
	 of the proof of the claim. $W_4$ has a $\frac{1}{5}(3,2,1)$ point. $W_8\rightarrow\cdots\rightarrow W_4$ is the economic resolution.\par
	 Finally the discrepancies of exceptional divisors follows from direct computation.
\end{proof}

\begin{lem}\label{nqf}
	Assume that $V\rightarrow W$ is a $w$-morphism. Let $X$ be a $\Q$-factorization of $W$ and let $Y$ be a $\Q$-factorization of $V$.
	Assume that $exc(V\rightarrow W)$ contains only one non-Gorenstein point $P$ which is of index $r+1$ and either $r=1$, $X$ is smooth and
	$exc(X\rightarrow W)$ is irreducible, or $r>1$ and $X$ contains only one non-Gorenstein point. Let $F$ be an exceptional divisor over $P$
	such that $a(F,V)=1$. If $a(F,W)=1+\frac{1}{r}$, then $P$ is not $\Q$-factorial and $\cen_YF$ is a curve.
\end{lem}
\begin{proof}
	We run $K_Y$-MMP over $W$ and we may assume that $X$ is the minimal model. Then $Y\dashrightarrow X$ factorize into
	\[ Y\dashrightarrow Y_1\dashrightarrow \cdots\dashrightarrow Y_k\rightarrow X,\] such that $Y\dashrightarrow Y_k$ is a sequence of flips
	and $Y_k\rightarrow X$ is a divisorial contraction. Note that $Y_k\rightarrow X$ is a divisorial contraction to a curve if $r=1$, and
	a $w$-morphism if $r>1$.\par
	Assume that $P$ is $\Q$-factorial, then $P$ appears on $Y$ and the flipping curve of $Y\dashrightarrow Y_1$ passes through $P$. By the
	negativity lemma, we know that $a(Y_k,F)>a(Y,F)=1$. On the other hand, every singular point on $Y_k$ has index less than $r+1$ since
	otherwise there exists an exceptional divisor $G$ such that $a(G,Y)<a(G,Y_k)\leq\frac{1}{r+1}$, which is impossible. Hence
	$a(F,Y_k)\geq 1+\frac{1}{r}=a(F,W)=a(F,X)$. This leads to a contradiction since $\cen_{Y_k}F$ is contained in $exc(Y_k\rightarrow X)$.
\end{proof}
\begin{lem}\label{cax}
	Notation as in Lemma \ref{le8}. The $cAx/4$ point on $W_3$ is not $\Q$-factorial. Let $\bar{Z}\rightarrow W_3$ be a $\Q$-factorization
	and let $\bar{Z}\dashrightarrow \bar{Z}'$ be the corresponding flop. Then $\bar{Z}$ has two singular points. One of them is a
	$\frac{1}{4}(1,3,1)$ point and the other one is a $\frac{1}{2}(1,1,1)$ point. We have $\cen_{\bar{Z}}E_7$ is the flopping curve,
	$\cen_{\bar{Z}}E_6$ is the $\frac{1}{2}(1,1,1)$ point and $\cen_{\bar{Z}}E_i$ is the $\frac{1}{4}(1,3,1)$ point for $i=4$,
	$5$ and $8$.\par
	Moreover, let $\tl{Z}$ be a $\Q$-factorization of $W_4$. If $\tl{Z}\neq W_4$, then the corresponding flop $\tl{Z}\dashrightarrow \tl{Z'}$
	is at worst a smooth flop over a $cA_2$ point. Assume that $\bar{Z}$ is the minimal model of $\tl{Z}$ over $W_3$. The birational map
	$\tl{Z}_0=\tl{Z}\dashrightarrow \bar{Z}$ can be factorize into
	\[\vc{\xymatrix{ \tl{Z}_{(0,1,1,1)} \ar[d]_c^{E_7} & \tl{Z}_{(0,1,1,0)} \ar@{-->}[l] \ar[d]^w & & \\
		\tl{Z}_{(0,1,2)} \ar[d]_w^{E_6} & \tl{Z}_{(0,1,1)} & \tl{Z}_{(0,1,0)} \ar@{-->}[l] \ar[d]^w & \\
		\tl{Z}_{(0,2)} \ar[d]_w^{E_5} & & \tl{Z}_{(0,1)} & \tl{Z}_{(0,0)} \ar@{-->}[l] \ar[d]^w \\
		\tl{Z}_1 \ar[d]_w^{E_4} & & & \tl{Z}_0 \\ \bar{Z} & & & }}\]
	such that every dash arrow is an Atiyah flop.
\end{lem}
\begin{proof}
	The $cAx/4$ point is not $\Q$-factorial and $\cen_{\bar{Z}}E_7$ is the flopping curve by Lemma \ref{nqf}. By the table in Lemma \ref{le8}
	we know that $\bar{Z}$ contains a singular point of index $4$, but there is no discrepancy one exceptional divisor over this point.
	It follows that the index $4$ point should be a cyclic quotient point. Since there are two exceptional divisor of discrepancy $\frac{1}{2}$
	over $\bar{Z}$, $\bar{Z}$ contains another $\frac{1}{2}(1,1,1)$ point.\par
	We are going to factorize the flop $\bar{Z}\dashrightarrow \bar{Z'}$. By direct computation, one can see that the Gorenstein singularities
	on $W_4$ are at worst $cA_2$ singularities. Let $\tl{Z}\rightarrow W_4$ be a $\Q$-factorization. If $W_4$ is not $\Q$-factorial, let
	$\tl{Z}\dashrightarrow \tl{Z'}$ be the corresponding flop. Otherwise let $\tl{Z'}=\tl{Z}$. We know that $\tl{Z}\dashrightarrow \tl{Z'}$ is
	at worst a smooth flop over a $cA_2$ point.\par
	Now $\tl{Z}_{(0)}=\tl{Z}$ has only a $\frac{1}{5}(3,2,1)$ singularity and $\tl{Z}_{(k)}$ has a $\frac{1}{3}(1,2,1)$ and $\frac{1}{2}(1,1,1)$
	point. Hence $dep(\tl{Z})=dep(\tl{Z_k})+1$, so $k=1$. We know that $exc(\tl{Z}_{(0,0)}\rightarrow \tl{Z}_{(0)})=E_5$ and
	$a(E_5,\bar{Z})=\frac{1}{2}$. Hence $\tl{Z}_{(0,0)}\dashrightarrow \tl{Z}_{(0,1)}$ is a flop and $\tl{Z}_{(0,2)}\rightarrow \tl{Z}_{(1)}$
	is a $w$-morphism by Lemma \ref{lflip}.\par
	Since $\cen_{\tl{Z}_{(0,1)}}E_7$ is the $\frac{1}{3}(1,2,1)$ point, this point lies on the flipping curve of
	$\tl{Z}_{(0,1)}\dashrightarrow \tl{Z}_{(0,2)}$. We have that $\tl{Z}_{(0,1,0)}\dashrightarrow \tl{Z}_{(0,1,1)}$ is a flop 
	and $\tl{Z}_{(0,1,2)}\rightarrow \tl{Z}_{(0,2)}$ is a $w$-morphism by Lemma \ref{lflip}. $\tl{Z}_{(0,1,1)}$ has only two
	$\frac{1}{2}(1,1,1)$ point. One can check that $a(E_8,\tl{Z}_{(0,1,2)})=a(E_8,\tl{Z}_{(0,1,1)})=\frac{1}{2}$. This implies that 
	$\cen_{\tl{Z}_{(0,1,2)}}E_8$ is a $\frac{1}{2}(1,1,1)$ point which is not contained in the flipped curve.  Thus the flipping curve of
	$\tl{Z}_{(0,1,1)}\dashrightarrow \tl{Z}_{(0,1,2)}$ contains only a $\frac{1}{2}(1,1,1)$ point and the factorization of this flip
	is given by Lemma \ref{pago}.\par
	Finally every flop appears in the factorization is an Atiyah flop by Lemma \ref{flopness} and Lemma \ref{ati}
\end{proof}
\begin{lem}\label{cd3}
	Notation as in Lemma \ref{le8} and Lemma \ref{cax}. Let $Z$ be a $\Q$-factorization of $W_2$ and assume that $Z$ is a minimal model of 
	$\bar{Z}$ over $W_2$. Then $Z$ has a $cA/3$ singular point defined by \[ (xy+z^3+u^2=0)\subset\A^4_{(x,y,z,u)}/\frac{1}{3}(1,2,1,0).\]
	There is a unique $w$-morphism $\bar{Z}_1\rightarrow Z$ such that $\bar{Z}_1$ has a $\frac{1}{2}(1,1,1)$ point and a $\frac{1}{3}(2,1,1)$
	point. The $\frac{1}{2}(1,1,1)$ point is the center of $E_6$ and the $\frac{1}{3}(2,1,1)$ point is the center of $E_4$ and $E_5$.
	The center of $E_8$ on $Z$ is a curve.\par
	We have the following factorization
	\[\vc{\xymatrix{ \bar{Z}_{(0,1,1,1)} \ar[d]_c^{E_8} & \bar{Z}_{(0,1,1,0)} \ar@{-->}[l] \ar[d]^w & & \\
		\bar{Z}_{(0,1,2)} \ar[d]_w^{E_5} & \bar{Z}_{(0,1,1)} & \bar{Z}_{(0,1,0)} \ar@{-->}[l]\ar[d]^w & \\
		\bar{Z}_{(0,2)} \ar[d]_w^{E_4} & & \bar{Z}_{(0,1)} & \bar{Z}_{(0,0)} \ar@{-->}[l]\ar[d]^w  \\
		\bar{Z}_1\ar[d]_w^{E_3} & & & \bar{Z}_0 \\ Z & & & }}\]
	such that every dash arrow is an Atiyah flop. 
\end{lem}
\begin{proof}
	Notice that a general elephant of a singular point on $Z$ have at most $5$ components, hence $Z$ do not have $cD/3$ singularity.
	Thus $Z$ has only a $cA/3$ singular point. Because there are exactly one exceptional divisor of discrepancy $\frac{1}{3}$, two
	exceptional divisor of discrepancy $\frac{2}{3}$ and one exceptional divisor of discrepancy $1$ over this singular point, we may assume that
	this point is of the following form  \[ (xy+z^3+u^2=0)\subset\A^4_{(x,y,z,u)}/\frac{1}{3}(1,2,1,0)\] by \cite[Proposition 3.4]{me}.
	We know that $\bar{Z}_k\rightarrow Z$ is a $w$-morphism and there is only one $w$-morphism over this $cA/3$ point.
	One can verify that $\bar{Z}_k$ has a $\frac{1}{2}(1,1,1)$ point and a $\frac{1}{3}(2,1,1)$ point. Thus $dep(\bar{Z}_k)=3=dep(\bar{Z})-1$
	and so $k=1$. Notice that $\bar{Z}$ contains two singular points, one is a $\frac{1}{4}(1,3,1)$ point and the other one is a
	$\frac{1}{2}(1,1,1)$ point. The flipping curve of $\bar{Z}_0\dashrightarrow \bar{Z}_1$ passes through only one of these two points
	since the flipped curve on $\bar{Z}_1$ (which is the proper transform of the flopping curve of $Z\dashrightarrow Z'$ on $\bar{Z}_1$)
	passes through only one singular point. Lemma \ref{le8} implies that $\cen_ZE_8$ is a curve, hence the flipping curve should passes
	through the $\frac{1}{4}(1,3,1)$ point and $a(E_6,\bar{Z}_1)=a(E_6,\bar{Z}_0)=\frac{1}{2}$. Hence the $\frac{1}{2}(1,1,1)$ point on
	$\tl{Z}_1$ is the center of $E_6$. Now the factorization of $\bar{Z}_0\dashrightarrow \bar{Z}_1$ can be contracted using the same method
	in the proof of Lemma \ref{cax}.
\end{proof}
\begin{lem}\label{ce2}
	Notation as in Lemma \ref{le8} and Lemma \ref{cd3}. We have that $Y$ has a $cA/2$ singularity with depth equals to $2$ or $3$ and $Z_{(1)}$
	has also a $cA/2$ singularity if $Z_{(1)}\rightarrow Y$ is a $w$-morphism, in the case that $dep(Y)=3$. We have $\cen_{Y}E_6$ is a curve on
	$Y$. $E_2$, $E_3$ (resp. $E_2$, $E_3$ and $E_4$) appears on the feasible resolution of $Y$ when $dep(Y)=2$ (resp. $dep(Y)=3$), Center of
	$E_5$ on the feasible resolution is a curve contained in $E_3$ and Center of $E_4$ is a curve contained in $E_2$ and not contained in $E_3$
	if $dep(Y)=2$.\par
	Moreover, we have the following factorization
	\[ Y\leftarrow Z_{(1)}\leftarrow Z_{(0,k_{(0)})}\dashleftarrow\cdots\dashleftarrow Z_{(0,1)} \dashleftarrow Z_{(0,0)}
		\rightarrow Z_{(0)}=Z\dashrightarrow Z'\] such that
	$Z_{(0,k_{(0)})}\rightarrow Z_{(1)}\rightarrow Y$ and $Z_{(0,0)}\rightarrow Z_{(0)}$ are $w$-morphisms, $Z_{(0,0)}\dashrightarrow Z_{(0,1)}$
	is an Atiyah flop, $k_{(0)}=3$ if $dep(Y)=3$ and $k_{(0)}=4$ if $dep(Y)=2$, $Z_{(0,1)}\dashrightarrow Z_{(0,2)}$ is given by Lemma
	\ref{3flip} (1) and other flips have the factorization in Lemma \ref{pago}.
\end{lem}
\begin{proof}
	Obviously $Y$ can not have $cE/2$ singularities because a general elephant of a singular point on $Y$ has at most $6$ components.
	Notice that there is only one exceptional divisor of discrepancy $\frac{1}{2}$ over the singular point of $Y$. Furthermore, let
	$Z_k\rightarrow Y$ be the $w$-morphism, then every non-Gorenstein point on $Z_k$ should have index $2$ since $Z$ has only a $cA/3$
	singular point. By the classification in \cite{h1,h2} and by a direct computation one can check that $Y$ can not have $cAx/2$ or $cD/2$
	singularities. Thus the singular point of $Y$ is a $cA/2$ point and the general elephant of this point has at worst $A_5$ singularities
	(\cite[(6.4)]{re2}. Moreover, there exists one discrepancy one exceptional divisor over this point. This implies that the singular point
	on $Y$ is of the form \[ (xy+z^2+u^n=0)\subset\A^4_{(x,y,z,u)}/\frac{1}{2}(1,1,1,0),\] with $n=2$, $3$. One can see that
	$dep(Y)=n=2$ or $3$.\par
	By Lemma \ref{nqf}, we know that $\cen_YE_6$ is a curve. $E_2$ and $E_3$ appear in the feasible resolution of $Y$ because they are
	only two exceptional divisors with discrepancy less than or equal to one. The $w$-morphism $Z_k\rightarrow Y$ extracts $E_2$. Since
	$\cen_ZE_8$ is a curve and $a(E_8,Z_k)=1$, $\cen_{Z_k}E_8$ is a curve. We have $a(E_4,Z_k)=1$. If $dep(Y)=2$ then $Z_k$ has only a
	$\frac{1}{2}(1,1,1)$ singular point, so $\cen_{Z_k}E_4$ is a curve. If $dep(Y)=3$, then $Z_k$ has a non-cyclic quotient $cA/2$ point and the
	center of $E_4$ should be this point. Finally since $a(E_5,Y)=2$, $E_5$ should be a curve contained in $E_3$.\par
	Let $C_{Z_{(0)}}$ be the flipping curve of $Z_{(0)}\dashrightarrow Z_{(1)}$ and let $C_{Z_{(0,0)}}$ be the proper transform of $C_{Z_{(0)}}$
	on $Z_{(0,0)}$. Note that $Z_{(0,0)}$ has two singular points. One of them is a $\frac{1}{3}(1,2,1)$ point and the other one is a
	$\frac{1}{2}(1,1,1)$ point. The explicit equation of the singular point on $Z_{(0)}$ is given by Lemma \ref{cd3}. According to Mori's local
	classification \cite[Page 243]{mo}, we know that the $C_{Z_{(0,0)}}$ won't pass through the $\frac{1}{3}(1,2,1)$ point. 
	We claim that $C_{Z_{(0,0)}}$ do not pass through any singular point and so $Z_{(0,0)}\dashrightarrow Z_{(0,1)}$ is a smooth flop.\par
	Indeed, assume that $C_{Z_{(0,0)}}$ passes through the $\frac{1}{2}(1,1,1)$ point. If $Z_{(0,0)}\dashrightarrow Z_{(0,1)}$ is a flop,
	then the non-$\Q$-factorial point of $U_{(0,0)}$ has exactly one exceptional divisor of discrepancy $\frac{1}{2}$ and one exceptional
	divisor of discrepancy $1$. This implies that this non-$\Q$-factorial point in defined by
	\[(xy+z^2+u^2=0)\subset\A^4_{(x,y,z,u)}/\frac{1}{2}(1,1,1,0).\] However this singularity is $\Q$-factorial by \cite[Proposition 2.2.7]{ko2},
	which leads to a contradiction. Hence $Z_{(0,0)}\dashrightarrow Z_{(0,1)}$ is a flip. The flipped
	curve is the center of $E_6$ and is contained in $exc(Z_{(0,k_{(0)})}\rightarrow Z_{(1)})=E_3$. In this case we have
	$a(E_6,Z_{(0,k_{(0)})})=a(E_6,Z_{(0,1)})=1$, hence this curve appears on $Z_{(0,k_{(0)})}$ and is contracted by
	$Z_{(0,k_{(0)})}\rightarrow Z_{(1)}$. This is impossible since $\cen_YE_6$ is a curve.\par
	So we know that $Z_{(0,0)}\dashrightarrow Z_{(0,1)}$ is a smooth flop. We need to say that the flipping curve of
	$Z_{(0,1)}\dashrightarrow Z_{(0,2)}$ do not connect the two singular points on $Z_{(0,1)}$. If it is true, then the flip is of the type
	semistable $IA+IA$ \cite[Theorem 2.2]{km2}. Let $H_{Z_{(0)}}$ be a general elephant of the flip $Z_{(0)}\dashrightarrow Z_{(1)}$. Then
	$H_{Z_{(0,1)}}$ contains the flipping curve of $Z_{(0,1)}\dashrightarrow Z_{(0,2)}$. On the other hand, the flipping curve intersects the
	flopped curve of $Z_{(0,0)}\dashrightarrow Z_{(0,1)}$ non-trivially, hence $H_{Z_{(0,1)}}$ intersects the flopped curve non-trivially. Since
	$H_{Z_{(0,1)}}$ intersects trivially to the flopped curve, it contains the flopped curve. Thus $H_{Z_{(0,0)}}$ contains $C_{Z_{(0,0)}}$
	and $H_{Z_{(0)}}$ contains $C_{Z_{(0)}}$. However it is impossible by the classification \cite[Theorem 2.2]{km2} and the fact that there are
	only one $cA/3$ singular point on $Z_{(0)}$.\par
	Hence the flipping curve of $Z_{(0,1)}\dashrightarrow Z_{(0,2)}$ passes through only one singular point. Assume first that $dep(Y)=3$.
	In this case $\cen_{Z_{(0,k_{(0)})}}E_4$ is a $\frac{1}{2}(1,1,1)$ point. It follows that $k_{(0)}=3$ and
	$Z_{(0,1)}\dashrightarrow Z_{(1)}$ can be factorized into two flips followed by a $w$-morphism 
	\[Z_{(0,1)}\dashrightarrow Z_{(0,2)}\dashrightarrow Z_{(0,3)}\rightarrow Z_{(1)}.\]
 	The two flipped curves correspond to $E_5$ and $E_6$. Since $\cen_{Z_{(0,3)}}E_5$ is a curve contained in $E_3$ and
 	$Z_{(0,3)}\rightarrow Z_{(1)}$ contracts $E_3$ to a point, the flipped curve of $Z_{(0,2)}\dashrightarrow Z_{(0,3)}$ corresponds to $E_6$.
 	Thus $Z_{(0,1)}\dashrightarrow Z_{(0,2)}$ is a flip around the $\frac{1}{3}(1,2,1)$ point and its factorization is given by
 	Lemma \ref{3flip} (1) since the flipped curve on $Z_{(0,2)}$ contains $\cen_{Z_{(0,2)}}E_4$, which is a $\frac{1}{2}(1,1,1)$ point.
 	Now $Z_{(0,2)}\dashrightarrow Z_{(0,3)}$ is a flip around a $\frac{1}{2}(1,1,1)$ point, and its factorization is given by Lemma \ref{pago}.\par
 	Now assume that $dep(Y)=2$. One has that every possible flip $Z_{(0,i)}\dashrightarrow Z_{(0,i+1)}$ is a flip around only one singular point.
 	Let $Z_{(0,j-1)}\dashrightarrow Z_{(0,j)}$ be the flip around a $\frac{1}{2}(1,1,1)$ point such that the flipped curve corresponds to $E_6$.
 	We want to say that $j=k_{(0)}$. Indeed, let $L_i=\cen_{Z_{(0,i)}}E_6$, then we have $K_{Z_{(0,j)}}.L_j=1$ by Lemma \ref{pago}. If
 	$j\neq k_{(0)}$ then $K_{Z_{(0,k_{(0)})}}.L_{k_{(0)}}<K_{Z_{(0,j)}}.L_j=1$, hence $K_{Z_{(0,k_{(0)})}}.L_{k_{(0)}}\leq0$ since
 	$Z_{(0,k_{(0)})}$ is smooth. However it is impossible because $L_{k_{(0)}}$ intersects $exc(Z_{(0,k_{(0)}}\rightarrow Y)$ non-trivially
 	and the image of $L_{k_{(0)}}$ on $Y$ is $\cen_YE_6$, which is a $K_Y$-trivial curve. Hence the last flip
 	$Z_{(0,k_{(0)}-1)}\dashrightarrow Z_{(0,k_{(0)})}$ is a flip around a $\frac{1}{2}(1,1,1)$ point.\par
	The flip $Z_{(0,1)}\dashrightarrow Z_{(0,2)}$ is a flip around a $\frac{1}{3}(1,2,1)$ point. We claim that $dep(Z_{(0,2)})=dep(Z_{(0,1)})-1$. 
	Hence $k_{(0)}=4$, the factorization $Z_{(0,1)}\dashrightarrow Z_{(0,2)}$ is given by Lemma \ref{3flip} (1) and
	$Z_{(0,2)}\dashrightarrow Z_{(0,3)}$ is also a flip around a $\frac{1}{2}(1,1,1)$ point.\par
	Assume that $dep(Z_{(0,2)})=dep(Z_{(0,1)})-2$, then $k_{(0)}=3$. Let $\Gamma_i=\cen_{Z_{(0,i)}}E_5$. Then $\Gamma_2$ is a flipped curve
	so $K_{Z_{(0,2)}}.\Gamma_2>0$. We have seen before that $K_{Z_{(0,3)}}.L_3=1$. One has $\phi\st K_Y=K_{Z_{(0,3)}}-\frac{1}{2}E_2-E_3$
	and $\phi\st K_Y.L_3=0$ where $\phi$ denotes the morphism $Z_{(0,3)}\rightarrow Y$. Hence $E_3.L_3\leq 1$. Since $\Gamma_3$ is a smooth curve
	contained in $E_3$, $\Gamma_3$ meets $L_3$ at at most one point transversally. This will imply that
	$K_{Z_{(0,2)}}.\Gamma_2-K_{Z_{(0,3)}}.\Gamma_3\leq 1$, hence $K_{Z_{(0,3)}}.\Gamma_3\geq 0$. However it is impossible since
	$\Gamma_3\subset E_3$ is contracted by $Z_{(0,3)}\rightarrow Z_{(1)}$.\par
	Finally we need to check that the flop $Z_{(0,0)}\dashrightarrow Z_{(0,1)}$ is an Atiyah flop. First assume that $dep(Y)=3$ and we have
	$k_{(0)}=3$. Let $\Xi_i$ be the flipping curve of $Z_{(0,i)}\dashrightarrow Z_{(0,i+1)}$. One can see that the proper transform of
	$\Xi_2$ on $Z_{(0,1)}$ is not the flopped curve of $Z_{(0,0)}\dashrightarrow Z_{(0,1)}$ because the flopped curve do not pass through
	the $\frac{1}{2}(1,1,1)$ point. This implies that $\Xi_1$ and $\Xi_2$ both appear on $Z_{(0,0)}$. Let $F_{Z_{(0,0)}}=E_3$ be the exceptional
	divisor of $Z_{(0,0)}\rightarrow Z_{(0)}$, then we have $F_{Z_{(0,i)}}$ is generically smooth along $\Xi_i$. This implies that
	$F_{Z_{(0,1)}}$ is smooth in the smooth locus of $Z_{(0,1)}$ by applying Lemma \ref{ati1} twice. Hence $Z_{(0,0)}\dashrightarrow Z_{(0,1)}$
	is an Atiyah flop by Lemma \ref{ati2}.\par
	Now assume that $dep(Y)=2$ and we have $k_{(0)}=4$. As before let $\Xi_i$ be the flipping curve of $Z_{(0,i)}\dashrightarrow Z_{(0,i+1)}$.
	Then $\Xi_i$ appears on $Z_{(0,0)}$ for $i=1$, $3$ by the same reason as in the previous case. We are going to show that $\Xi_2$ also appears
	on $Z_{(0,0)}$, or equivalently, the proper transform of $\Xi_2$ on $Z_{(0,1)}$ is not the flopped curve of
	$Z_{(0,0)}\dashrightarrow Z_{(0,1)}$. Let $C_{Z_{(0,1)}}$ be the flopped curve of $Z_{(0,0)}\dashrightarrow Z_{(0,1)}$ and let $C_{Z_I}$
	be the proper transform of $C_{Z_{(0,1)}}$ on $Z_I$, for all possible $I$. We are going to show that $C_{Z_{(0,2)}}\neq\Xi_2$.
	Assume that $C_{Z_{(0,2)}}=\Xi_2$. In this case we have $C_{Z_{(0,1)}}$ intersects $\Xi_1$ non-trivially. We know that $K_{Z_I}.C_{Z_I}=0$
	for $I=(0,1)$, $(0,1,0)$, $(0,1,1)$, $(0,1,1,0)$ and $(0,1,1,1)$. However, $K_{Z_{(0,1,2)}}.C_{Z_{(0,1,2)}}<0$ since $C_{Z_{(0,1)}}$
	intersects $\Xi_1$ non-trivially. On the other hand, we have $Z_{(0,1,2)}=Z_{(0,2,0)}$ and the proper transform of $\Xi_2$ on $Z_{(0,2,0)}$
	is $K_{Z_{(0,2,0)}}$-trivial, since $Z_{(0,2,0)}\dashrightarrow Z_{(0,2,1)}$ is a flop. Thus $C_{Z_{(0,2)}}\neq \Xi_2$ and $F_{Z_{(0,1)}}$
	is smooth in the smooth locus of $Z_{(0,1)}$, which implies that $Z_{(0,0)}\dashrightarrow Z_{(0,1)}$ is an Atiyah flop

\end{proof}
\begin{pro}\label{e8flop}$ $
	Assume that $X\dashrightarrow X'$ is a simple smooth flop over $W$ such that $W$ has a $cE_8$ singularity. We have the factorization
	\[ X\leftarrow Y_{(1)}\leftarrow Y_{(0,2)}\dashleftarrow Y_{(0,1)}\dashleftarrow Y_{(0,0)} \rightarrow Y_{(0)}=Y\dashrightarrow Y'\]
	\[= Y'_{(0)}\leftarrow Y'_{(0,0)}\dashrightarrow Y'_{(0,1)}\dashrightarrow Y'_{(0,2)}\rightarrow Y'_{(1)}\rightarrow X' \]
	such that $Y_{(0,2)}\rightarrow Y_{(1)}\rightarrow X$ is a sequence of blowing-up smooth curves, $Y_{(0,0)}\dashrightarrow Y_{(0,1)}$ is
	an Atiyah flop. $Y_{(0,1)}\dashrightarrow Y_{(0,2)}$ is given by Lemma \ref{pago} (reps. Lemma \ref{3flip} (1)) if $dep(Y)=2$
	(resp. $dep(Y)=3$). $Y\dashrightarrow Y'$ is a flop given in Lemma \ref{ce2} and the diagram $Y'\dashrightarrow X'$ is symmetric.
\end{pro}
\begin{proof}
	By Lemma \ref{ce2} we know that $Y$ contains only a $cA/2$ point with depth $2$ or $3$. We use the notation as in Convention \ref{cn}. 
	Notice that there are only one exceptional divisor with discrepancy less than one over $Y=Y_{(0)}$, hence $k_{(0)}=1$. Thus
	$Y_{(1)}\rightarrow X$ is a blowing-up a smooth curve. If $dep(Y)=2$, then $Y_{(0)}\dashrightarrow Y_{(1)}$ is given by
	Lemma \ref{3flip} (3).\par
	Now assume that $dep(Y)=3$. Since $Y$ has only a $cA/2$ singular point, by the same argument as in the last paragraph of the proof of Lemma
	\ref{3flip} we can show that $Y_{(0,0)}\dashrightarrow Y_{(0,1)}$ is a flop. It is a smooth flop by Lemma \ref{flopness}. Now
	$dep(Y_{(0,1)})=2$ and $Y_{(0,1)}$ contains only a $cA/2$ point. Since $exc(Y_{(0,0)}\rightarrow Y_{(0)})=E_2$ and $\cen_{Y_{(1)}}E_2$
	is a curve, $Y_{(0,2)}\rightarrow Y_{(1)}$ is a blowing-up a smooth curve. Hence $Y_{(0,2)}$ is smooth and
	$Y_{(0,1)}\dashrightarrow Y_{(0,2)}$ is given by Lemma \ref{3flip} (3). Finally $Y_{(0,0)}\dashrightarrow Y_{(0,1)}$ is an Atiyah flop
	by Lemma \ref{ati}.
\end{proof}
\begin{proof}[Proof of Theorem \ref{ssthm}]
	The statement (1) is well-known. (2) follows from Proposition \ref{cd} and (3) follows from Proposition \ref{ce6}, Proposition \ref{ce7}
	and Proposition \ref{e8flop}.
\end{proof}

\end{document}